\documentclass[11pt, reqno]{amsart}
\usepackage{amsmath,amssymb,amsfonts,amscd,hyperref,color}
\usepackage[latin1]{inputenc}

\usepackage[abs]{overpic}
\usepackage{verbatim}

\newcommand{\Hmm}[1]{\leavevmode{\marginpar{\tiny%
$\hbox to 0mm{\hspace*{-0.5mm}$\leftarrow$\hss}%
\vcenter{\vrule depth 0.1mm height 0.1mm width \the\marginparwidth}%
\hbox to
0mm{\hss$\rightarrow$\hspace*{-0.5mm}}$\\\relax\raggedright #1}}}

\newtheorem{theorem}{Theorem}

\newtheorem{thm}{Theorem}[section]

\newtheorem{lem}[thm]{Lemma}
\newtheorem{lemma}[thm]{Lemma}
\newtheorem{pro}[thm]{Proposition}

\theoremstyle{definition}
\newtheorem*{defi}{Definition}
\newtheorem{eg}[thm]{Example}
\newtheorem*{rem}{Remark}

\numberwithin{equation}{section}
\newcommand{\Z}{{\mathbb Z}}
\newcommand{\R}{{\mathbb R}}

\newcommand{\N}{{\mathbb N}}

\newcommand{\al}{{\alpha}}

\newcommand{\ph}{{\varphi}}
\newcommand{\eps}{{\varepsilon}}

\newcommand{\si}{{\sigma}}

\newcommand{\lm}{{\lambda}}
\newcommand{\Deg}{{\mathrm{Deg}}}
\newcommand{\supp}{{\mathrm{supp}\,}}

\begin{document}
\title[Agmon estimates on graphs]
{Agmon estimates for Schr\"odinger operators on graphs}
\author[M.~Keller]{Matthias Keller}
\address{M.~Keller,  Institut f\"ur Mathematik, Universit\"at Potsdam
\\14476  Potsdam, Germany}
\email{matthias.keller@uni-potsdam.de}
\author[F.~Pogorzelski]{Felix Pogorzelski}
\address{F.~Pogorzelski, Institut f\"ur Mathematik, Universit\"at Leipzig
\\04109 Leipzig, Germany}
 \email{felix.pogorzelski@math.uni-leipzig.de}
\date{\today}
\begin{abstract}
We prove decay estimates for generalized eigenfunctions of  discrete Schr\"odinger operators on weighted infinite graphs in the spirit of Agmon.
\\[2mm]
\noindent  2000  \! {\em Mathematics  Subject  Classification.}
Primary  \! 39A12, 35P05; Secondary   31C20, 31C25, 35B09,  35R02, 47B39.\\[2mm]
\end{abstract}
\maketitle
\setcounter{section}{-1}

\setcounter{section}{-1}
\section{Introduction}
Decay properties of  (generalized) eigenfunctions of Schrödinger operators are of great significance in spectral theory and  mathematical physics. In this paper, we study the decay of such eigenfunctions on graphs in the spirit of the work of Agmon from 1982, \cite{Ag82}.\smallskip

The spectral theory of Schrödinger and Laplace operators has been studied extensively in the last decade. This concerns on the one hand questions of spectral types, see e.g \cite{ABS,AF00,Br07,DFS,FHS,HN,KS14,KS15,KLWa2,Nag04,NW02} and the beautiful survey \cite{MoharWoess89} on earlier works.  On the other hand, various results on spectral bounds, purely discrete spectrum and asymptotics of the corresponding eigenvalues have been established recently, see e.g. \cite{BG15,BGK15,GHKLW,Gol14,HKSW,K10,KLSW,KMP,KPP,LSS,LSS2,Woj2}. \smallskip

While there is a huge industry studying decay properties of eigenfunctions for random operators, see e.g. the monograph \cite{AW_book} and references therein, relatively little is known about the decay of eigenfunctions in the setting of deterministic operators on general discrete graphs.  In contrast, for quantum graphs, Agmon estimates on the decay of eigenfunctions have been studied in \cite{AP17,AP18,HM18}. Furthermore, there are sophisticated considerations on Agmon estimates on the discrete lattice $ \Z^{d} $ in \cite{KR08,KR14,KR16}. Moreover, there are also applications of Combes-Thomas estimates for the decay of eigenfunctions of Schrödinger operators on $ \Z^{d} $, cf.\@ \cite{Man}.\smallskip

The results in this paper are based on the study of  Hardy and Rellich inequalities. Especially, we build on the earlier work which establishes these type of inequalities on graphs \cite{KePiPo2,KPP_Rellich} (see \cite{DFP,Fi00,FP11,Ro18,Ro20,Ro21} for related work on  operators in the continuum setting and Dirichlet forms). \smallskip

Moreover, when talking about generalized eigenfunction one should mention the relation to Shnol' type theorems \cite{Shnol,BLS,CFKS,BD19} which were recently studied in the context of graphs \cite{BP17,HK}. Furthermore, a general perspective is taken in \cite{LT} on studying generalized eigenfunctions on graphs. 
\smallskip

As an appetizer, we illustrate two special cases of our results in the case of bounded combinatorial graphs. Let a discrete set $ X $ with an adjacency relation $ \sim $ be given.  The combinatorial Laplacian   $ \Delta $ acts on functions $ f:X\to \R $ as
\begin{align*}
	\Delta f(x)=\sum_{y\sim x} (f(x)-f(y)).
\end{align*}
 For simplicity we assume in this introduction that there is a uniform upper bound $ D $ on the vertex degree, i.e.\@ the number of neighbors for any given vertex. 
Under this assumption the restriction of $ \Delta $ to the real Hilbert space $\ell^{2}(X)  $ gives rise to a bounded operator. Moreover, consider a  potential $ q:X\to \R$ such that a restriction of the Schrödinger operator $$ H= \Delta+q  $$ to $ \ell^{2}(X) $ yields a positive self-adjoint operator. 

First we  present a very simple qualitative consequence of our results. 
We say that a function $ u $ is \emph{decreases exponentially} (in the $ \ell^{2} $-sense) if there is $ r>0 $ such that 
\begin{align*}
	ue^{r|\cdot|}\in \ell^{2}(X),
\end{align*}
where $ |x| $ denotes the combinatorial graph distance of a vertex $ x $ to some fixed vertex. One of our results,  Theorem~\ref{t:main5} which is proven below, includes the following theorem as a  special case.

\begin{theorem}\label{t:intro1}
Every generalized eigenfunction of $H $  to an generalized eigenvalue strictly below the essential spectrum  of $ H $ which satisfies 	$ \|u\vert_{B_{n+1}\setminus B_{n-1}} \|\to 0$ as $ n\to\infty $,   decreases exponentially.\footnote{We learned from discussions with David Damanik and Sylvain Gol\'enia, that this statement can alternatively be  concluded by means of a Shnol' type theorem, see e.g. \cite[Theorem~4.8]{HK}, a Combes-Thomas estimate \cite[Theorem~10.5]{AW} and the holomorphic spectral calculus \cite[Sections III.4-5]{kato_book}.}
\end{theorem}

We next turn to a more quantitative result. 
The philosophy of Agmon estimates is that generalized eigenfunctions which do not grow too fast indeed decay rapidly \cite{Ag82}. 

To this end assume that $ H $ satisfies a \emph{Hardy inequality} with a  \emph{Hardy weight} $ w:X\to [0,\infty) $ outside of a finite set $ K\subseteq X $, that is
\begin{align*}
	\langle H\ph,\ph\rangle \ge \langle w\ph,\ph\rangle
\end{align*}
for all functions $ \ph $ which are  compactly supported outside of $ K $. 

For example for Laplacian $ \Delta_{\Z^{d}} $, $ d\ge3 $, on the standard Cayley graph of $ \Z^{d} $ it was shown 
in \cite{KePiPo2} that there is a Hardy weight $ w_{\Z^{d}} $ with asymptotic
$$  w_{\Z^{d}} (x) = \frac{(d-2)^{2}}{4}\frac{1}{\|x\|_{\Z^{d}}^{2}} +O\left(\frac{1}{\|x\|_{Z^{d}}^3}\right),$$
where $ \|\cdot\|_{\Z^{d}} $ denotes the Euclidean norm on $ \Z^{d} $.

Given a Schrödinger operator $ H $ and a strictly positive Hardy weight $w  $, we define a metric as an analogue of the so called \emph{Agmon metric} in the continuum setting
$$ \varrho (x,y)=\inf_{x=x_{0}\sim\ldots \sim x_{n}=y}\sum_{i=1}^{n}(D\wedge w(x_{i})\wedge w(x_{i+1}))^{\frac{1}{2}} ,$$
where $ D $ is an upper bound on the vertex degree.
Furthermore, for some $ o\in X $, we denote $ \varrho=\varrho(o,\cdot) $. 
The following quantitative consequence of Theorem~\ref{c:main2} can be referred to as an Agmon estimate.

\begin{theorem}\label{t:intro2} Let $ u $ be a generalized eigenfunction of $ H $ to  $ \lambda $ and let $ w $  be a Hardy weight of $ H-\lambda $ outside of a compact set. 
If  the Agmon metric is a complete metric space and if  there is $ \al>0 $ with $ \al^{2}e^{\al}<8 $ and
\begin{align*}
		u\in \ell^{2}(X,we^{-(\al/\sqrt{D})\varrho}),
\end{align*}	
then,
	\begin{align*}
		u\in \ell^{2}(X,we^{(\al/\sqrt{D})\varrho}),
	\end{align*}
where $ D $ is an upper bound on the vertex degree.
\end{theorem}
A further application of  the theory developed in this paper is found in \cite{BGK2} for the special case of planar graphs with uniformly decaying curvature.


As mentioned above, the two theorems above are special cases of our main results which are proven in the setting of positive Schrödinger operators on locally finite weighted graphs. 
This general set up is discussed in the next section, Section~\ref{s:setup}, and we introduce some of the basic tools which are well known and needed throughout the paper. 

Afterwards in Section~\ref{s:approx}, we introduce the notion of approximable functions. This notion is  more flexible and provides a much more general version of the growth assumption in Theorem~\ref{t:intro1}. Indeed, the definition of approxim\-able functions is the only part of this section needed for the results and most of the applications and Section~\ref{s:Rellich} and Section~\ref{s:appl}.  As the concept seems to be interesting in its own right we study it in more depth in Section~\ref{s:approx}.

Our techniques rely on Rellich inequalities which are deduced from Hardy inequalities in Section~\ref{s:Rellich}. This was already studied in \cite{KPP_Rellich} on $ C_{c}(X) $, however here we have to do this on a larger space of functions. For strongly local Dirichlet forms this was done by Robinson \cite{Ro18}, although we take a somewhat different route in our approach. This allows us to prove a Rellich inequality, Theorem~\ref{p:main} and an Agmon estimate, Theorem~\ref{t:main}. These results are the main body of the work and all further applications which are then studied in Section~\ref{s:appl} rely on them. 

Among these applications in Section~\ref{s:appl} is Theorem~\ref{t:main2}, which is a discrete version of a classical theorem of Agmon from \cite{Ag82} and which includes Theorem~\ref{t:intro2} above as a special case.
Next, we prove a general result, Theorem~\ref{t:main5}, for eigenfunctions below the essential spectrum which includes Theorem~\ref{t:intro1} as a special case.  Then, we study graphs which satisfy a Cheeger inequality or are sparse and have purely discrete spectrum. The eigenvalue asymptotics of such graphs were studied in \cite{BG15,BGKLM20,Gol14} and we prove statements about the decay of the corresponding eigenfunctions here. Finally, we draw the connection to optimal Hardy inequalities which were studied in \cite{KePiPo2} for graphs building on ideas from corresponding considerations in the contiuum setting \cite{DFP,FP11}.

\section{Set up and Toolbox}\label{s:setup}
\subsection{Graphs,  operators, forms and intrinsic metrics}\label{Subsection-Formal}

Let $ b $ be a graph over a discrete measure space $ (X,m) $. That is, $ X $ is an infinite countable set equipped with the discrete topology, $ m:X\to(0,\infty) $ extends to a measure via $ m(A)=\sum_{x\in A}m(x)$ and  $b:X\times
X\to[0,\infty)$ is  a symmetric function with zero diagonal such that
$$\sum_{y\in  X}b(x, y)<\infty \quad \textup{ for all } x\in X.$$
For $ x,y\in X $ with $ b(x,y)>0 $, we write $ x\sim y $. We call a graph \emph{locally finite} if $ \#\{y\sim x\} <\infty$ for all $ x\in X $. Although this assumption appears in the main results, we will only assume it in this paper when the assumption is actually used.

For $W\subseteq X$, we denote by $C(W)$ (resp., $C_c(W)$) the space
of real valued functions on $W$ (resp., with compact support in
$W$). The space
$C(W)$  will be considered as a subspace of  $C(X)$ by extending functions by zero on $X\setminus W$. We denote the space of $m$-square summable functions by $\ell^{2}(X,m)$ with norm $ \|\cdot\|_{m} $ and we
denote the space of bounded functions by $\ell^{\infty}(X)$ with norm $ \|\cdot\|_{\infty} $. 

\medskip
\noindent
For a function $ f\in C(X) $, we denote
\begin{align*}
|\nabla f|(x)=\left(\frac{1}{2m(x)} \sum_{y\in X}b(x,y)(f(x)-f(y))^{2}\right)^{1/2},
\end{align*}
where the value $ \infty $ is allowed.
For a function $ q:X\to\R $, let $ h=h_{b,q}:C_{c}(X)\to \R $ be given by
\begin{align*}
h(\ph)=\sum_{ X}(|\nabla \ph|^{2}+q\ph^{2})m,
\end{align*}
which extends to a bilinear form by polarization. Here and also in the remainder of the paper, we use the short-hand notation $\sum_X f:= \sum_{x \in X} f(x) $ for $f \in C(X)$, whenever $ f\ge0 $ or the latter series converges absolutely.   

We identify a function $ w:X\to \R $ 
with the quadratic form on $ C_{c}(X) $ given by
$$  w(\ph)=\sum_{X}\ph^{2}wm ,\qquad \ph\in C_{c}(X).  $$
If $ w\ge 0 $ and
\begin{align*}
	 h\ge w \qquad\mbox{on }C_{c}(X),
\end{align*}
 we say that $ h $ satisfies a \emph{Hardy inequality} with \emph{Hardy weight} $ w $.

For the following, we assume that $ h $ is bounded from below on $ \ell^{2}(X,m) $. Then, $ h $ is a closable   and  we denote the closure also by $ h=h_{b,q} $ with domain $ D(h) $. 
Decomposing $ q $ into positive and negative part $ q=q_{+}-q_{-} $, we denote
$$  h_{+}=h_{b,q_{+}}  $$ and let
\begin{align*}
V=\{q \in C(X)\mid q_{-}\leq (1-\eps) h_{+}+C\mbox{ on }C_{c}(X) \mbox{ for some }\eps>0, C>0 \}.
\end{align*}
If $ q\in V $, we have by the KLMN Theorem, see \cite[Theorem X.17]{RSII} 
\begin{align*}
D(h_{+})=D(h)
\end{align*}
and
\begin{align*}
\eps h_{+}-C\leq h\leq h_{+}\qquad\mbox{on }D(h).
\end{align*}

The corresponding semi-bounded self-adjoint operator on $ \ell^{2}(X,m) $ is denoted by $ H $ with domain $ D(H) $ and can be seen by \cite{Schm3} to be a restriction of  the \emph{formal
	Schr\"odinger operator} $\mathcal{{H}}=\mathcal{H}_{b,q,m}$ acting as
\begin{align*}
{\mathcal{H}}f(x):=\frac{1}{m(x)}\sum_{y\in X}b(x,y)(f(x)-f(y))+q(x)f(x) \quad \mbox{ for } x \in X
\end{align*}
on the space $ \mathcal{F}:=\mathcal{F}(X) $ 
\begin{align*}
{\mathcal{F}}(X):=\{f\in C(X) \mid \sum_{y\in X}b(x,y)|f(y)|<\infty\mbox{ for
	all } x\in X\}.
\end{align*}
We call $ u\in \mathcal{F} $ a \emph{(sub/super-)solution} to $ \lm\in \R $ if
\begin{align*}
(\mathcal{H}-\lm)u=0 \qquad((\mathcal{H}-\lm)u\le 0, \,(\mathcal{H}-\lm)u\ge 0)
\end{align*}
Especially, (sub/super-)solutions to $ \lm=0$
are called \emph{(sub/super-)harmonic}.

\medskip
\noindent
A pseudo-metric is a symmetric function $X\times X\to[0,\infty)$ with zero diagonal that satisfies the triangle inequality. The \emph{jump size} $s$ of a pseudo-metric $d$ is given by
\begin{align*}
s:=\sup\{d(x,y)\mid x,y\in X, b(x,y)>0\}.
\end{align*}
An \emph{intrinsic metric} $d$ is a pseudo-metric that satisfies
\begin{align*}
\sum_{y\in X}b(x,y)d(x,y)^{2}\leq m(x)
\end{align*}
for all $x\in X$. For a set $ U\subseteq X $, we denote
\begin{align*}
d(U,x)=\inf_{y\in U}d(x,y),\qquad x\in X.
\end{align*}
In view of the introduction, we note that whenever the vertex degree $\Deg(x)=\frac{1}{m(x)} \sum_{y\in X}b(x,y) $, $ x\in X $ is bounded by a constant $ D $, then the combinatorial graph distance divided by $ \sqrt{D} $  is an intrinsic metric.

Next, we collect various results which we need along the way to prove the main results of the paper.

\subsection{The ground state transform and an Allegretto-Piepenbrink type theorem}
The tools introduced in this section are well established. However, for the readers convenience we recall these results and even sketch some of the proofs. We start with a Green's formula and then discuss a so called ground state transform. Afterwards we present a consequence of the Allegretto-Piepenbrink  theorem for the bottom of the  spectrum $\lm_{0}(H)$ and the bottom of the essential spectrum $\lm_{0}^{\mathrm{ess}}(H)$ of the operator $ H $.

\noindent

We recall a rather general version of Green's formula. For related versions of this formula see \cite{Do84, Woj1, KL1, HK, KePiPo1} and references therein.

\begin{lemma}[Green's formula]\label{l:Green} Let $ a:X\times X\longrightarrow \R $ be given  such that $ \sum_{y\in X}|a(x,y)| <\infty$ for all $ x\in X $. Then, for all $ u:X\longrightarrow\R $ such that $ \sum_{y\in X}|a(x,y)u(y)| <\infty$ for each $x \in X$ and all $ \ph\in C_{c}(X) $
\begin{align*}
\frac{1}{2}	\sum_{x,y\in X}a(x,y)(u(x)-u(y )(\ph(x)-\ph(y)) =\sum_{x\in X} u(x)\sum_{y\in X}a(x,y)(\ph(x)-\ph(y)).
\end{align*}
Specificially, we have for all $ \ph,\psi\in C_{c}(X) $
\begin{align*}
h(\psi,\ph )=\sum_{X}\psi \mathcal{H}\ph m.
\end{align*}	
\end{lemma}
\begin{proof}This follows by direct algebraic manipulation using the assumptions to ensure the absolute convergence of the involved sums.
\end{proof}

The following ground state transform for graphs is proven in \cite[Proposition~4.8]{KePiPo1},  see also \cite{FLW,FS08,Gol14,HK} and references therein.
 For a positive function $ v $, we denote for $ \ph\in C_{c}(X) $
\begin{align*}
|\nabla_{v} \ph|^{2}(x)=\frac{1}{2m(x)}\sum_{y\in X}b(x,y)v(x)v(y)(\ph(x)-\ph(y))^{2}
\end{align*}
for $ x\in X $, where the value $ +\infty $ is allowed. However, $ |\nabla_{v}\ph|^{2}<\infty $ for all $ \ph\in C_{c}(X) $ whenever $ v\in \mathcal{F} $ which is easily seen by the Cauchy-Schwarz inequality and the summability assumption on $b $.

\begin{pro}[Ground state transform]\label{p:GST} For all $ v \in \mathcal{F}$ and $ \ph\in C_{c}(X) $ we have
	\begin{align*}
\frac{1}{2}	\sum_{x,y\in X}b(x,y)v(x)v(y)(\ph(x)-\ph(y))^{2}+\sum_{x\in X}( v\ph^{2})(x)\mathcal{H}v(x) m(x)= h(v\ph).
	\end{align*}
	In particular, if $ v \in  \mathcal{F} $ is a positive supersoluton to $ \lm\in\R $  on $ U $, then for all $ \ph\in C_{c}(U) $
\begin{align*}
\sum_{ X}|\nabla_{v}\ph|^{2}m\leq (h-\lm )(v\ph).
\end{align*}
If furthermore, $ v $ is a solution then equality holds.
\end{pro}
\begin{proof}
	By a direct calculation we have for all $ x\in X $
	\begin{align*}
	v(x)\mathcal{H}(v\ph)(x)=(v\ph)(x)\mathcal{H}v(x)+\frac{1}{m(x)}\sum_{y\in X} b(x,y)v(x)v(y)(\ph(x)-\ph(y)).
	\end{align*}
	Thus, the statement now follows by summing over $ x $ after multiplying by $\ph(x) m(x) $ and then applying Green's formula to the term on the left hand side.
\end{proof}

Next, we present a version of  the so-called Allegretto-Piepenbrink theorem. Related results for graphs and Dirichlet forms are found in \cite{BLS,Do84,FLW,HK,KePiPo1} and see also references therein.

\begin{pro}[Allegretto-Piepenbrink theorem]\label{p:AP}
	Let $\lambda \in \R$. 
	\begin{itemize}
		\item[(a)] One has $ \lambda\leq \lambda_{0}(H) $ if and only if there is a strictly positive supersolution to $ \lambda $.
			\item [(b)]	Suppose that the graph is locally finite.
		For all $ \lm < \lm_{0}^{\mathrm{ess}}(H)$ there exists a finite set $ K\subseteq X $ and a strictly positive solution to $ \lm $ and 
		$ h-\lm \ge0 $ on $ C_{c}(X\setminus K) $. If there is a strictly positive supersolution to $\lambda$ on $X \setminus K$ for some finite set $K$, then $\lambda \leq \lambda_0^{\mathrm{ess}}(H)$. 
	\end{itemize}

\end{pro}

\begin{proof}
	(a) By  \cite[Theorem~3.8]{Schm3}, $H$ is a restriction of $\mathcal{H}$. Thus, the statement follows from the Allegretto-Piepenbrink theorem proven in 
	\cite[Theorem 4.2~(b)]{KePiPo1}. \\
	(b) This follows from \cite[Theorem~4.2 (c)]{KePiPo1}. 
	\end{proof}

\begin{rem}There is the question whether in the Allegretto-Piepenbrink theorem above one has a supersolution for $ \lambda=\lambda_{0}^{\mathrm{ess}}(H). $ This can indeed be achieved if there are only finitely many eigenvalues of finite multiplicity below $ \lambda_{0}^{\mathrm{ess}}(H) $. This is indeed a characterization and can be concluded from \cite[Theorem~1.3/Corollary~1.4]{Si11} and is shown in the case of manifolds in \cite{Dev12}.	
\end{rem}

\subsection{A Caccioppoli type inequality}

In this section we provide a Caccioppoli type inequality which is in one form or another known to experts as well.

The following auxiliary lemma is a variant of a corresponding results for  for Dirichlet forms from \cite[Lemma~3.5]{GHM}.
\begin{lemma}\label{l:Lipform}Assume $ q\in V $ and $ w\ge 0 $. Let $\psi\in \ell^{\infty}(X)$ be such that
	\begin{align*}
	|\nabla\psi|^2\leq w.
	\end{align*}
	Then, for $ u\in D(h)\cap \ell^{2}(X,wm)  $, we have
$ u\psi\in D(h)\cap \ell^{2}(X,wm)$.\end{lemma}
\begin{proof}
By elementary estimates and the assumption $ |\nabla\psi|^2\leq w $, we infer 	for $ v\in C_{c}(X) $
 \begin{align*}
	h(v\psi)\leq 2\|\psi\|_{\infty}^{2}h_{+}(v)+2\sum_{X}v^{2}|\nabla\psi|^2m\leq 2\|\psi\|_{\infty}^{2}h_{+}(v)+2\|v\|^{2}_{wm}.
	\end{align*}
    As $ q\in V $, we have $ D(h)=D(h_{+}) $ and, therefore, $ u \in D(h_{+}) $. 		
	Let  $ (u_{n}) $ be a sequence of compactly supported functions  which approximates $ u $ in form norm of $ h_{+} $. 		
	Set $ v=u_{k}-u_{n} $ for $ k,n\in \N $. We use $ q\in V $  together with the estimate in the beginning of the proof  to obtain
	\begin{align*}
 {\varepsilon}	h_{+}(v{\psi})-C\|v{\psi}\|_{m}\leq h(v\psi)\leq 
	2\|\psi\|_{\infty}^{2}h_{+}(v)+2\|v\|^{2}_{wm}
	\end{align*}
	for suitable choices for $\varepsilon, C > 0$. This shows that $ (u_{n}\psi) $ is a Cauchy-sequence with respect to $ h $. By the boundedness of $ \psi $ the assumption that $ u_{n}\to u $ in $ \ell^{2}(X,m) $ we also have that  $ (u_{n}\psi) $ converges to $ u\psi $ in $ \ell^{2}(X,m) $. Thus,  $ u\psi\in D(h) $. Finally, boundedness of $ \psi $  and $ u\in\ell^{2}(X,wm) $ also yields that $ u\psi\in \ell^{2}(X,wm) $.
\end{proof}



Recall the discrete Leibniz rule for functions $ f,g\in C(X) $ and $ x,y\in X $
\begin{align*}
(fg)(x)-(fg)(y)=f(x)(g(x)-g(y))+g(y)(f(x)-f(y))
\end{align*}
Furthermore,  we denote for $ v\in \mathcal{F} $ and $ \psi\in C_{c}(X) $ and 
$ x \in X $
\begin{align*}
	\nabla v\cdot\nabla\psi (x): = \frac{1}{2m(x)} \sum_{y \in X} {b(x,y)}(v(x)-v(y))(\psi(x)-\psi(y))
\end{align*}
which is an absolutely converging sums by $ v\in \mathcal{F} $, $ \ph\in C_{c}(X) $ and summability of $  b$. We show a version of  a Caccioppoli type inequality which extends for example \cite[Lemma~3.4]{HKMW}.

\begin{pro}[Caccioppoli type inequality]\label{l:Cacciopolli} Assume one of the following assumptions:
	\begin{itemize}
		\item $u\in \mathcal{F}$ and $\psi\in C_{c}(X)$  or
		\item  $ u\in D(H) $, $\psi\in \ell^{\infty}(X)$ such that
		$ 		|\nabla\psi|^2\leq w  $ for some $ w\ge0 $ and there is a Hardy inequality $ h\ge w $ on $ C_{c}(X), $ as well as 
		$ q\in V $.
	\end{itemize}
	Then,
	\begin{align*}
	h(\psi u)\leq \sum_{X}|\nabla_{|u|} \psi|^{2} m  +\sum_{X}(\mathcal{ H} u)(\psi^{2}u)m.
	\end{align*}
\end{pro}
\begin{proof}
	By the Leibniz rule, we compute formally 
	\begin{align*}
	\sum_{X}|\nabla \psi u|^{2}m&=    \sum_{X}u\nabla\psi\cdot\nabla (\psi u)m+    \sum_{X}\psi\nabla u\cdot\nabla (\psi u)m\\
	&= \sum_X u \nabla \psi \cdot \nabla (\psi u) m - \sum_X \psi u \nabla \psi \cdot \nabla u \, m + \sum_X \nabla u \cdot \nabla(\psi^2 u) m \\
	&\leq {\sum_{X}|\nabla_{|u|}\psi|^{2}m }+\sum_{X}\nabla u\cdot \nabla(\psi^{2}u)m.
	\end{align*}
	Adding $ \sum q (\psi u)^{2}m $ on both sides, we obtain the statement after applying  Green's formula.
	
	Indeed, it is not hard to check by the Cauchy-Schwarz inequality that all sums above converge absolutely if  $ u\in \mathcal{F} $ and $ \ph\in C_{c}(X) $. Moreover, the Hardy inequality $ h\ge w $  on $ C_{c}(X) $ extends easily to $ D(h)\supseteq D(H) $ by approximation, so whenever  $u\in D(H)$ we have $ u\in \ell^{2}(X,wm) $. Furthermore, if $\psi$ is as assumed as in the second case then we find by the lemma above that $ u\psi,u\psi^{2}\in D(h) $. Since $|\nabla \psi|^2 \leq w$ by assumption, one gets 
	$ \sum_X |\nabla_{|u|}\psi|^{2} m\leq \|u\|_{wm}^2 $
	by the Cauchy-Schwarz inequality and with
	$ u\in \ell^{2}(X,wm) $
	one concludes that all terms converge absolutely.
\end{proof}

\subsection{Form bounds outside of compact sets}

The following technical lemma shows how we can extend a lower form bound outside of a compact set $ K $ to a lower form bound on the whole space by changing the form on a combinatorial neighborhood of $ K $.  
For a set $ U\subseteq X $ and a graph $ b $, we define the \emph{combinatorial neighborhood} of $ U $
by
$$   N(U):=U\cup\{y\in X\mid y\sim x\mbox{ for some }x\in U  \}. $$
\begin{lem}\label{l:positive}
	Assume $ h\ge 0$ on $ C_{c}(X\setminus K) $ for some compact set $ K\subseteq X $. Let
	\begin{align*}
		\lm_K=\inf_{\ph\in C_{c}(K),\|\ph\|=1}h(\ph)
	\end{align*} 
	and  $ b_K: X\to [0,\infty) $
	\begin{align*}
		b_{K}=\frac{1}{m}\left(1_{K}\sum_{y\in X\setminus K}b(\cdot,y) +1_{X\setminus K} \sum_{y\in K}b(\cdot,y)\right).
	\end{align*}
	Then, $ \chi=2b_{K}-\lambda_{K}1_{K} $ is supported on $ N(K) $ and $$  h+\chi\ge 0  \qquad \mbox{on }  C_{c}(X) . $$
\end{lem}
\begin{proof}
	Clearly, $ b_{K} $ is supported on $ N(K) $ so $ \chi $ is supported on $ N(K) $. We have  
	\begin{align*}
	h(\varphi 1_K, \ph 1_{X \setminus K}) &= - \frac{1}{2} \sum_{x \in K} \sum_{y \in X \setminus K} b(x,y) \varphi(x) \varphi(y) \\
	& \quad - \frac{1}{2} \sum_{x \in X \setminus K} \sum_{y \in K} b(x,y) \varphi(x) \varphi(y),
	\end{align*}
	and by means of the Young inequality we obtain
	\begin{align*}
	h(\ph 1_{K},\ph1_{X\setminus K})& \geq -b_{K}(\ph)
	\end{align*}
	for all $ \ph\in C_{c}(X) $. For $ \ph\in C_{c}(X) $, we obtain since $ h(\ph 1_{X\setminus K})\ge 0 $
	\begin{align*}
	h(\ph)=h(\ph 1_{X\setminus K})+ 2h(\ph 1_{K},\ph 1_{X\setminus K})+ h(\ph 1_{ K})\ge  -2b_{K}(\ph) + \lm_{K}1_{K}(\ph),
	\end{align*}
	which finishes the proof.
\end{proof}
This auxiliary lemma is the place where the assumption of local finiteness which is invoked later enters the scene. Indeed, $ N(K) $ is only guaranteed to be compact for compact $ K $ if the graph is locally finite. This then corresponds to the support of the function $ \chi $ above which is needed to be in $ C_{c}(X) $ for the considerations below.

\section{Approximable functions}\label{s:approx}
In this section we introduce the concept of approximable functions. Next to functions in the operator domain, these functions allow for the Agmon estimates we present later on.

We first define the concept of approximable  functions in Subsection~\ref{sec:approximable} and give a sufficient criterion in Proposition~\ref{p:growth}. With this in mind one can jump forward to the main results and their proofs in the next chapters. However, as the concept of approximable functions does not seem to be standard, we explore it further in this section and relate it to other more well-known concepts. With the help of intrinsic metrics, we give a  criterion for functions to be approximable which is of a more geometric nature and is in the spirit of Agmons original work. Moreover, we discuss that approximable functions $ u $ for which $ h(u) $ exists in a suitable sense are already in the form domain. Finally, we provide links to the existence of ground states and criticality.

\subsection{Approximable and weakly approximable functions}
\label{sec:approximable}

\begin{defi}[Approximable function] A function $u\in \mathcal{F}(X)$ is called \emph{approximable} if there is a sequence $(\ph_{n})$ of functions in $C_{c}(X)$ such that  $ 	0\leq \ph_{n}\nearrow 1 $ as $   n\to\infty $ and
	\begin{equation*}
 \lim_{n\to\infty}\sum_{X}u^2|\nabla \ph_{n}|^{2}m=0.
	\end{equation*}
Moreover, $u\in \mathcal{F}(X)$ is called \emph{weakly approximable} if there is $(\ph_{n})$  in $C_{c}(X)$ such that $0\leq \ph_{n}\nearrow 1$, $ n\to\infty $ and $$\lim_{n\to\infty}\sum_{X}|\nabla_{|u|} \ph_{n}|^{2}m=0.$$
We call $\ph_{n}$  \emph{(weak) approximating cut-off functions} for $u$.
\end{defi}
In the context of  local Dirichlet forms the condition on $ u $ to be approximable appears in \cite[Theorem~1.1, Condition III]{Ro18}. Our main results build only on the notion of approximable functions. So, for the proofs in the subsequent sections it is sufficient to observe that approximable functions are weakly approximable, see Lemma~\ref{l:AtoWA}.

However, these two different notions may be of independent interest as they are an example of how the discrete and the continuum setting differ. Indeed, in the continuum setting there is only one corresponding notion by the virtue of the chain rule. In the lemma we show the relation between the two notions in the graph setting. 

For the remainder of this subsection, we will investigate these notions more deeply. As mentioned above, the reader only interested in the Agmon estimates may safely skip this material and come back to it at some later point.

To this end we call a function $u$ \emph{non-oscillatory} if 
$$\sup_{x\sim y} \frac{|u(x)|}{|u(y)|}<\infty.$$

\begin{lemma}\label{l:AtoWA}
Approximable functions are weakly approximable and weakly approximable, non-oscillatory functions are approximable.
\end{lemma}
\begin{proof} The first statement follows directly from Young's inequality. As for the second statement we observe that there is $C>0$ such that for all $\ph\in C_c(X)$ we have
	$|\nabla_{|u|}\ph|^2\leq Cu^2|\nabla \ph|^2.$ This finishes the proof.
\end{proof}

The monotonicity assumption on the (weak) approximating functions $\ph_n$ as well as $0\leq \ph_n\leq 1$ will be subsequently used later. However, it is not a restriction as it can be seen from the next lemma.

\begin{lemma}\label{l:mon} Let  $u\in \mathcal{F}(X)$ be a function such that there is a sequence $(\psi_{n})$ of functions in $C_{c}(X)$ such that
		$\psi_{n}\to 1$, $ n\to\infty $ and $$\lim_{n\to\infty}\sum_{X}u^2|\nabla \psi_{n}|^{2}m=0,\qquad\left(\mbox{respectively, }\lim_{n\to\infty}\sum_{X}|\nabla_{|u|} \psi_{n}|^{2}m=0\right).$$
then, $u$ 	is (weakly) approximable, i.e., there is a (weak) approximating sequence $(\ph_{n})$ of functions in $C_{c}(X)$ such that
$0\leq \ph_{n}\nearrow 1$, $ n\to\infty $ and $$\lim_{n\to\infty}\sum_{X}u^2|\nabla \ph_{n}|^{2}m=0,\qquad\left(\mbox{respectively, }\lim_{n\to\infty}\sum_{X}|\nabla_{|u|} \ph_{n}|^{2}m=0\right).$$
\end{lemma}
\begin{proof}We only show the statement for  approximable functions. The weakly approximable case  follows analogously.

Without loss of generality we can assume that $0\le \psi_n\le 1$ since we have $|\nabla(0\vee\psi_n\wedge 1)|\leq|\nabla\psi_n|$.

To turn $(\psi_n)$ into a monotone sequence, we choose a sequence $(\eta_n)$ such that $\operatorname{supp} \psi_n = \operatorname{supp} \eta_n$, $0 <\eta_n \leq \psi_n$ on $\operatorname{supp} \psi_n$,  $\eta_n \to 0$ as $n \to \infty$, and $\sum_{X}u^2|\nabla \eta_n|^{2}m \to 0$ as $n \to \infty$ and we set 
$\ph_n:= \psi_n - \eta_n$. Then, $0\leq \ph_n<\psi_n\leq 1$ for all $n\in\N$. Now since $\psi_n\to 1$, $\eta_n\to0$, $\ph_n<\psi_n$ and all $\ph_n$ have finite support, we can extract a subsequence $(\ph_{n_k})$ that converges to $1$ montonically increasingly. Furthermore, 
\begin{align*}
\sum_{X}u^2|\nabla \ph_{n_k}|^{2}m\leq 2\sum_{X}u^2|\nabla \psi_{n_k}|^{2}m+2\sum_{X}u^2|\nabla \eta_{n_k}|^{2}m\to 0
\end{align*}
as $k\to\infty$.
This concludes the proof. 
\end{proof}

\subsection{Approximability and volume growth}

Next we give a criterion for a function to be  approximable with the help of  an intrinsic metric $d$. We denote the balls about a fixed vertex $ o\in X $ with radius $ r \geq 0 $ by
\begin{align*}
B_{r}=\{x\in\mid d(o,x)\leq r \}.
\end{align*}
We say that the balls are compact for $ d $
if $ B_{r} $ are finite for all $ r\ge0 $.

\begin{pro}\label{p:growth}If the intrinsic metric $ d $ has finite jump size $ s $ and compact balls of finite radius, then  every $u\in \mathcal{F}(X)$ with
	\begin{align*}
	\liminf_{n\to\infty}\sum_{B_{n+\eps+s}\setminus B_{n-s}}u^{2}m=0
	\end{align*}
	for some $ \eps>0 $ is approximable.
\end{pro}
	\begin{proof}Let $\ph_{n}=(1- d(B_{n},\cdot)/\eps)_{+}$ for $ \eps>0 $. Since $d$ has finite jump size $s$ and since the balls of finite radius are finite, we get $\ph_{n}\in  C_{c}(X)$, $0\leq\ph_{n}\nearrow 1$ pointwise and $ |\nabla \ph_{n} |^{2}$ is supported on $ B_{n+\eps+s}\setminus B_{n-s} $. Moreover, by the basic inequality $ ab\leq (a^{2}+b^{2})/2 $ and the intrinsic metric property we obtain
		\begin{align*}
		\sum_{X}u^2|\nabla \ph_{n}|^{2}m= \sum_{B_{n+\eps+s}\setminus B_{n-s}}u^{2}|\nabla \ph_{n}|^{2}m\leq \frac{1}{\eps}
		\sum_{B_{n+\eps+s}\setminus B_{n-s}}u^{2}m.
		\end{align*}
		Note that we used here that for $x,y$ with $b(x,y) > 0$, we have 
			\begin{align*}
			 & \eps \left| \left( 1 - \frac{d(B_n, x)}{\eps} \right)_{+} - \left( 1 - \frac{d(B_n, y)}{\eps} \right)_{+} \right| 
			  \,\, \leq \,\,  \left| d(B_n, x) - d(B_n, y)\right| \leq d(x,y). 
			\end{align*}					
		Thus, we can extract a subsequence of approximating cut-off functions for $u$ from $(\ph_{n})$.
	\end{proof}

	A sequence $(F_n)$ consisting of non-empty finite subsets of (the vertices of) a graph $b$ on $(X,m)$ is called a {\em F{\o}lner sequence} if 
	\[
	\lim_{n \to \infty} \frac{b\big(\partial{F_n}\big)}{m(F_n)} = 0,
	\] 
	where $ \partial A=A\times (X\setminus A) $ for $A \subseteq X$ and  we define
	$b(\partial A) :=\sum_{(x,y)\in \partial A} b(x,y)$. A F{\o}lner sequence is called {\em nested} if $F_n \subseteq F_{n+1}$ for all $n \in \N$. 
	The {\em vertex boundary of a set $A \subseteq X$} is defined as 
	\begin{align*}
	\partial_V A &:= \big\{ x,y\mid x\in A,y\in X\setminus A,  x\sim y  \big\}.
	\end{align*}
	
	\begin{pro} \label{p:folnerapprox}
		Suppose that $b$ is a graph on $(X,m)$ admitting a nested F{\o}lner sequence $(F_n)$ such that $\bigcup_n F_n = X$. Then every function $u \in C(X)$ for which there exist a constant $C > 0$ and $n_0 \in \N$ such that 
		\[
		u_{|\partial_V F_n} \leq C\cdot {m(F_n)}^{-1/2} \quad \mbox{ for all } n \geq n_0
		\] 
		is approximable. 
	\end{pro} 

	\begin{proof}
		Let $u$ be a function satisfying the boundedness condition given in the statement of the proposition. 
		For $n \geq n_0$ we define $\varphi_n := 1_{F_n}$. Since $(F_n)$ is nested and $\bigcup_{n} F_n = X$, we have $0 \leq \varphi_n \nearrow 1$ as $n \to \infty$.
		 Moreover, we observe that $ \varphi_n(x)-\ph_n(y) = 0$ if $x \notin \partial_V F_n$ and $y \sim x$.
		 Hence, for each $n  \geq n_0$ we get 
		\begin{align*}
			\sum_{X} u^2 |\nabla \varphi_n|^2 m &= \sum_{x \in \partial_V F_n} u^2 |\nabla \varphi_n|^2 m \leq C^2 \cdot {\frac{2}{m(F_n)}} \cdot b(\partial F_n).
			\end{align*} 
			The observation that $(F_n)$ is a F{\o}lner sequence concludes the proof.
		\end{proof}
	
	\begin{eg}
		Let $b$ be the  graph with standard weights on $\Z^d$ with the counting measure $m=1$. Precisely, for $x,y \in \Z^d$ we have $b(x,y) = 1$ if $\sum_{i=1}^d|x_i - y_i| = 1$ and $b(x,y) = 0$ otherwise. 
		For $n \in \N$, we write $F_n := \big\{ (x_i)_{i=1}^d \in \Z^d\mid |x_i| \leq n \mbox { for all } 1 \leq i \leq d \big\}$. We have $m(F_n) = \# F_n = (2n+1)^d$, where $\# A$ denotes the cardinality of a set $A$. It is readily checked that $\partial_V F_{n} \subseteq F_{n+1} \setminus  F_{n-1} $, and therefore,  
		\[
		b\big( \partial F_n \big) \leq \big( (2n+3)^d - (2n-1)^d \big).
		\]
		Hence, $F_n$ is a F{\o}lner sequence. 
		 Now, the above proposition shows that every function $u \in C(\Z^d)$ for which there is a constant $C > 0$ such that 
		 \[
		 u_{|F_{n+1} \setminus F_{n-1}} \leq C \, (2n+1)^{-d/2} \quad \mbox{ for all } n \in \N
		 \]
		 is approximable. Moreover, for $d \geq 3$, consider the function \[
		 G:\Z^d \to [0, \infty), \quad G(x) := \begin{cases}
		  \kappa |x|^{2 - d}, & x \neq 0 \\
		  1, & x = 0
		 \end{cases}
		 ,
		 \] 
		 where $\kappa > 0$ is a constant. Then $G$ is approximable if $d \geq 4$. Indeed, for $n \in \N$ and $x \in F_{n+1} \setminus F_{n-1}$, we have $n \leq |x| \leq \sqrt{d} \cdot (n+1)$. It follows that $G(x) \leq \kappa \cdot n^{2-d}$
		 for those $x$. For $d \geq 4$, this yields 
		 \[
		 G(x) \leq \kappa \cdot  n^{2-d} \leq \kappa \cdot n^{-d/2} \leq 4^{d/2}\kappa \cdot (4n)^{-d/2} \leq 4^{d/2} \kappa \cdot (2n + 1)^{-d/2}
		 \]
		 for all $n \geq 1$ and all $x \in \Z^d \setminus F_{n-1}$. By what we showed above the function $G$ is approximable. We point out that the asymptotic behavior of $G$ as $|x|\to \infty$ coincides with the asymptotic behavior of the Green function of the graph $b$ as $|x| \to \infty$, cf.\@ \cite{Uch98}. It follows from this that the Green function for $b$ is approximable as well if $d \geq 4$.  
		\end{eg}

\subsection{Approximable functions in the form domain}
We now show that functions $ u $, where we can explicitly give sense to $ h(u) $, are approximable if and only if they are in $ D(h) $.
\begin{lemma} Let  $ u \in\ell^2(X,m)$  be such that $ \sum_{X}(|\nabla u|^{2}+|q|u^{2})m<\infty $. Then, $ u $ is approximable if and only if $ u\in D(h) $. 
\end{lemma}
\begin{proof}Assume that  $ u $ is approximable.
	Let $ (\ph_{n}) $ be a sequence of approximating cut-off functions. Then, $ u_{n}=\ph_{n}u \in C_{c}(X)$ for all $n\in\N$ and for $k\ge n$ we have
	\begin{align*}
	h&(u_k-u_{n})=h(u(\ph_k-\ph_{n}))\\
	&\leq \sum_{X}u^2|\nabla(\ph_k-\ph_{n})|^{2}m+
	\sum_{X}(\ph_k-\ph_{n})^{2}(|\nabla u|^{2}+|q|u^{2})m
	\\&\leq
	2\sum_{X}u^2|\nabla\ph_k|^{2}m+2\sum_{X}u^2|\nabla\ph_n|^{2}m+
	\sum_{X}(1-\ph_{n})^{2}(|\nabla u|^{2}+|q|u^{2})m.
	\end{align*}
	The terms on the right hand side become arbitrarily small for large $k,n$ where use by Lebesgue's dominated convergence theorem for the third term. Moreover, $u_n\to u$ in $\ell^2(X,m)$ by monotone convergence.
	Hence, $(u_n) $ is an $h$-Cauchy-sequence and, therefore, $u \in D(h) $.
	
	On the other hand assume that $ u\in D(h) $. Then, there is $ (u_{n}) $ in $ C_{c}(X) $ such that $ h(u-u_{n})\to 0 $. We can write $ u_{n}=\ph_{n}u $ with $ \ph_{n}\in C_{c}(X) $ and can assume that $ 0\leq \ph_{n}\leq 1 $ and $ \ph_{n}\to 1 $, confer \cite[Lemma~3.2]{KePiPo1} or \cite[Lemma~6.6]{KLW_book}.
	Then, $ u-u_{n}=(1-\ph_{n})u$ and using the discrete Leibniz rule and Young's inequality, we obtain for some constant $ 0 <  \eps < 1  $ and some corresponding  $ C(\eps) \geq 1$
	\begin{align*}
	h&((1-\ph_{n})u)\\=& 	\sum_{X}\left( u^{2}|\nabla \ph_{n}|^2
	+(1-\ph_{n})^2|\nabla u|^{2}+q(1-\ph_{n})^{2}u^2\right)m\\
	& - \sum_{x,y\in X}b(x,y)u(x)(1-\ph_{n})(y) (u(x)-u(y))(\ph_{n}(x)-\ph_{n}(y)) \\
	\ge& (1-\eps)	\sum_{X} u^{2}|\nabla \ph_{n}|^2m
	-C(\eps)\sum_{X} \left(({1-\ph_{n}})^2|\nabla u|^{2}+|q|(1-\ph_{n})^{2}u^2\right)m.
\end{align*}
Hence, there is $ C>0 $ such that	
	\begin{align*}
		\sum_{X}u^{2}|\nabla \ph_{n}|^2m\leq C\left( h((1-\ph_{n})u)
		+\sum_{X}	(1-\ph_{n})^{2}\left(|\nabla u|^{2} +|q|u^2\right)m\right)\to 0
	\end{align*} 
 as $ n\to\infty $ where the convergence of the second term is ensured by Lebesgue's  theorem as  $ \ph_{n}\leq 1 $, $ \ph_{n}\to 1 $ and $ \sum_{X}(|\nabla u|^{2}+|q|u^{2})m<\infty $.	
\end{proof}

The following example shows that the finite energy assumption on  $u  $ is indeed necessary to conclude $ u\in D(h) $.
\begin{eg} Let $ X=\N $ and $ m(n)=1/n^2 $, $ n\in \N $. Furthermore, let $ b $ be symmetric such that $ b(k,k+1)=1/k $ and $ b(k,k')=0 $ for $ |k-k'|\ge 2 $. Then, the function $ u(k)=(-1)^k $ is in $ \ell^{2}(X,m) $ and approximable 
	with approximating cut-off functions $ \ph_{n}=1_{[1,n]\cap \N} $ but 
	$$  \sum_{k\in \N}|\nabla u|^2m=4\sum_{k\in\N}\frac{1}{k}=\infty. $$
	So, $ u $ cannot be in $ D(h) $.
\end{eg}

\subsection{Domination by solutions of minimal growth and criticality}
In this subsection we  show that functions which are dominated by a positive solution of minimal growth are weakly approximable. Specifically, a positive function $ v\in \mathcal{F}(X) $ is said to be a \emph{solution  of minimal growth at  infinity} if  $ \mathcal{H}v=0 $ on $ X\setminus K $  for some finite set  $ K\subseteq X $ and for all $ u\in \mathcal{F} $ such that $ \mathcal{H}u\ge0 $ on $ X\setminus L $ for some finite set $ K\subseteq L \subseteq X $ there is $ C> 0 $ such that 
\begin{align*}
Cv\leq u\qquad\mbox{ on }X\setminus L.
\end{align*}
A positive solution $v$ of minimal growth at infinity which satisfies $\mathcal{H}v=0$ on $X$ is called a ground state.

A  form $h$ is called {\em subcritical} if it admits a non-trivial Hardy weight, i.e. there is a non-trivial  $w \geq 0$ such that $h\ge w $ on $C_c(X)$. A positive form $h$ is called {\em critical} if it is not subcritical.  In the special case $ q=0 $ criticality  is referred to as recurrence.

\begin{lemma}
	If $v$ is a strictly positive solution of minimal growth at infinity  and if $|u| \leq v$ on $X \setminus K$ for some finite $ K\subseteq X $, then $u$ is weakly approximable. If $u$ is additionally non-oscillatory, then $u$ is approximable.  
\end{lemma}
\begin{proof}
Since $ v>0 $	and $ \mathcal{H}v $ is a solution outside a finite set, there is a finitely supported  function $ \chi $ such that $ (\mathcal{H}+\chi)v=0 $ on $ X $. Thus, $ h+\chi\ge0 $ on $ C_{c}(X) $ by the Allegretto-Piepenbrink theorem. Moreover, we show that $ h+\chi $ is critical,  i.e.\@ if $h + \chi \geq w$ for some $w \geq 0$, then we have necessarily $w=0$. 

\medskip
\noindent
\emph{Claim. $ h+\chi $ is critical.}

\medskip
\noindent 
\emph{Proof of the claim.} 
If there is a non-trivial $ w\ge0 $ such that $ h+\chi\ge w $ on $ C_{c}(X) $, then by the 
Allegretto-Piepenbrink theorem, there is a strictly positive supersolution $ v' $ to $ (\mathcal{H}+\chi- w)v' \ge 0 $. Since $ v $ is a solution of minimal growth at infinity there is $ C_{L}  $ such that $ C_{L} v\leq  v'$ on $ X\setminus L $ for some finite $ L\subseteq X $. As $ L $ is finite and $v$ is strictly positive, there is a maximal $ C $ such that $$ C v\le  v' \qquad \mbox{ on }  X . $$ Then,  $ v''=v'-Cv $ satisfies
\begin{align*}
(\mathcal{H}+\chi)v'' =(\mathcal{H}+\chi)v'- C(\mathcal{H}+\chi)v=(\mathcal{H}+\chi)v'\ge v'w\ge 0
\end{align*}
As $ w $ is non-trivial and $v'$ is strictly positive, $ v'' $ is not harmonic and therefore $ v''\neq0 $.
As $ v $ has minimal growth at infinity, we get by the same argument as above the existence of some $ C'>0 $ such that
 $$ C' v\le  v'' \qquad \mbox{ on }  X .$$
Thus, 
\begin{align*}
(C+C')v\leq Cv+v''=v'.
\end{align*}
This contradicts the maximality of $C  $. Thus, $ w=0 $ and this proves the claim.

 So $h + \chi$ is critical and $v$ is a strictly positive solution to $(H + \chi) u = 0$. Hence, $v$ is the ground state for the form $h + \chi$. 
 By \cite[Theorem~5.3~(iv')]{KePiPo1}, there is a sequence $(v_n)$ in $C_c(X)$ such that $0 \leq v_n \leq v$ such that $ v_{n} \to v $  pointwise as $ n\to\infty $ and $ (h+\chi)(v_n) \to 0 $, i.e.\@ $(v_n)$ is a null-sequence for $h + \chi$.
 Define $\psi_n := v_n / v$ for $n \in \N$.  
 By $ |u|\leq v $ and the ground state transform of $ h+\chi $ we obtain
 \begin{align*}
 0\leq \sum_{X}|\nabla_{|u|} \psi_{n}|^{2}\leq \sum_{X}|\nabla_{v} \psi_{n}|^{2} 
 =
 (h+\chi)(\psi_{n}v) = (h+\chi)(v_n)  \to 0
 \end{align*}
 as $ n\to\infty $. Clearly, $\psi_n=v_n/v\to 1$ as $n\to\infty$ and, therefore, there exists an weak approximating sequence $(\ph_n)$ by Lemma~\ref{l:mon}
 above. Thus, $u$ is weakly approximable. If $u$ is also non-oscillatory, then it is approximable by Lemma~\ref{l:AtoWA}.
\end{proof}

Next, we discuss the connection of the existence of approximable functions with a concept called criticality. Recall that whenever a form $h$ is critical, then every  positive superharmonic function is a multiple of a unique strictly positive harmonic function, cf.\@ \cite[Theorem~5.3]{KePiPo1}. 

\medskip

The connection between existence of weakly approximable positive subharmonic functions and criticality is established next.

\begin{pro}
	Suppose that the graph is connected. Assume that
 $ h\ge0  $ and that there exists a non-trivial  positive subharmonic function $u$ that is weakly approximable. Then,  $ h $ is critical.
\end{pro}
\begin{proof} Since $ h\ge0 $ there exists a strictly positive superharmonic function $ v $ {by the Allegretto-Piepenbrink theorem, Proposition~\ref{p:AP}}. Furthermore,
let $(\ph_{n})$ be a sequence of weak approximating cut-off functions for a positive subharmonic function $ u $.
By the ground state transform, Proposition~\ref{p:GST}, the Caccioppoli type inequality, Lemma~\ref{l:Cacciopolli},  and the assumptions on $u$ we arrive at
\begin{align*}
0\leq \sum_X \big| \nabla_v(\ph_{n}uv^{-1}) \big|^2 m &{\leq}\,  h(\ph_{n}u) \leq {\sum_{X}|\nabla_{|u|}\ph_n|^{2}m } + \sum_{X}(\mathcal{H}u)(\ph_{n}^{2}u)m \\
&\leq {\sum_{X}|\nabla_{|u|}\ph_n|^{2}m }.
\end{align*}
Since $(\ph_{n})$ is a sequence of weak approximating cut-off functions we get  by  Fatou's lemma 
\begin{align*}
0 &\leq \sum_X \big| \nabla_v (uv^{-1}) \big|^2m  \leq\liminf_{n\to\infty}  \sum_X \big| \nabla_v(\varphi_n u v^{-1}) \big|^2 m    \\
&{\leq}\liminf_{n\to\infty} {\sum_{X}|\nabla_{|u|}\ph_n|^{2}m }=0,
\end{align*}
which implies that $(u/v)(x)=(u/v)(y)$ for all $x\sim y$.
Thus, $u/v=\mathrm{const}$ since we assumed that the graph is connected. Thus, $ u $ is harmonic and every positive superharmonic function $ v $ is a linear multiple of the non-trivial function $ u$.	
\end{proof}

\section{Rellich inequalities and Agmon estimates}
\label{s:Rellich}
The general Agmon estimates which we prove in this paper are based on the following estimate. A version in the  continuum goes back to Agmon \cite{Ag82}.

\begin{thm}[Rellich inequality]\label{p:main} 
	Assume $h\ge w$ on $C_{c}(X)$ with $ w\ge  0 $.
	Let $g$ be a positive bounded function that satisfies an eikonal inequality, i.e.\@
	there is a constant $ 0<\gamma<1 $ such that 
	\begin{align*}
		|\nabla g^{1/2}|^{2}\leq \gamma g w .
	\end{align*}
	Then, for every   
	\begin{itemize}
		\item[(a)] approximable (positive sub-)solution $ u\in \mathcal{F} $
		\item[(b)] solution $ u\in D(H) $ whenever $ q\in V $
	\end{itemize}
	of $ Hu=f $ for some $ f\in C(X)$ which is supported in $ \supp w $, we have
	\begin{align*}
		(1-\gamma)^{2}	\sum_{X}|u|^{2}g w m\leq \sum_{X}|f|^{2}gw^{-1}m.
	\end{align*}
\end{thm}

\begin{rem}
	Note that if $g$ is zero at one point, then by the eikonal inequality $g$ must vanish on the entire  connected component containing that point. In this situation, the statements of the theorem become trivial on the connected component.  
\end{rem}

\begin{proof} We first give the proof under the assumption (a) and explain how this has to be modified under assumption (b) afterwards.
	
(a) Let $ (\ph_{n}) $  be a sequence of approximating functions for  $ u $.
For $n>0$, define
\begin{align*}
 \psi_{n}=\ph_{n}g^{1/2}.
\end{align*}
Then,  the given Hardy inequality and by the Caccioppoli type inequality, Lemma~\ref{l:Cacciopolli}, yield
\begin{align*}
\sum_{X}\psi^{2}_{n}u^{2}wm
\leq& h(\psi_{n} u)\\
\leq& \sum_{X}(\mathcal{H} u) (\psi^{2}_{n}u)m +\sum_{X}|\nabla_{|u|}\psi_{n}|^{2}m\\
\leq &\sum_{X}f\psi^{2}_{n}um+\sum_{X}u^2|\nabla\psi_{n}|^{2}m,
\end{align*}
where the last inequality follows as $u$ is a (positive sub-)solution and Young's inequality applied to the second term.
We obtain by the Cauchy-Schwarz inequality as $ w>0 $ on  $ \supp f  $
\begin{align*}
\ldots\leq&\Big(\sum_{X}u^{2}\psi^{2}_{n} w m\Big)^{\frac{1}{2}} \Big(\sum_{X}f^{2}\psi^{2}_{n}w^{-1}m\Big)^{\frac{1}{2}} +\sum_{X}u^2|\nabla \psi_{n}|^{2}m.
\end{align*}
To finish the proof, we continue to estimate the last term on the right hand side. We use the Leibniz rule to expand $ |\nabla \psi_{n}|^{2}=|\nabla(\ph_{n}g^{1/2})|^{2} $ with
$0\leq \ph_{n}\leq 1 $, $ |g|\leq  C $ for some $ C>0 $ and   Young's inequality $(a+b)^{2}\leq (\frac{N}{N-1})a^{2}+Nb^{2}$, $ N> 0$ and obtain 
\begin{align*}
{\sum_{X} u^2 |\nabla\psi_{n}|^{2}m}
\leq& \frac{N}{N-1}\sum_{X}u^{2}\ph_{n}^{2}|\nabla g^{1/2}|^{2}m+
N C\sum_{X}u^2|\nabla \ph_{n}|^{2}m
\\
 \leq& \gamma\frac{N}{N-1}\sum_{X}u^{2}\psi^{2}_{n} wm+NC\sum_{X}u^{2}|\nabla \ph_{n}|^{2}m,
\end{align*}
where we used the eikonal inequality.
Putting the last two estimates together, we arrive at
\begin{align*}
\sum_{X}u^{2}\psi^{2}_{n}wm\leq& \Big(\sum_{X}u^{2}\psi^{2}_{n}w m\Big)^{\frac{1}{2}} \Big(\sum_{X}f^{2}\psi^{2}_{n}w^{-1}m\Big)^{\frac{1}{2}}\\
&+ \gamma\frac{N}{N-1}\sum_{X}u^{2}\psi^2_{n} wm+
NC\sum_{X} u^2|\nabla \ph_{n}|^{2}m.
\end{align*}
Now, observe that the term $ \sum_{X}u^{2}\psi^{2}_{n}wm $ is finite as $ \psi_{n}=g\ph_{n} $ and $ \ph_{n} \in C_{c}(X)$. So, rearranging the inequality yields with $ c_{N}=(1-\gamma N/(N-1)) $
\begin{align*}
c_{N}	\left(\sum_{X}u^{2}\psi^{2}_{n}wm\right)^{\frac{1}{2}}\leq&  \Big(\sum_{X}f^{2}\psi^{2}_{n}w^{-1}m\Big)^{\frac{1}{2}}
	+ 
	NC\frac{\sum_{X} u^2|\nabla \ph_{n}|^{2}m}{\left(\sum_{X}u^{2}\psi^{2}_{n}wm\right)^{1/2}}.
\end{align*}
We first take the limit as $ n\to \infty $ and use monotone convergence and the approximibility with  $ \ph_{n} $. Finally,  taking the limit  $ N\to \infty $ yields the statement.

(b) For $ u\in D(H) $, we can follow the line of the above argument with $ \ph_{n}=1 $, i.e., $ \psi_{n}= \psi = g^{1/2} $.  To this end, we observe that $ u\psi\in \ell^{2}(X,wm) \cap D(h) $ which can be seen by applying Lemma~\ref{l:Lipform} whose assumptions are verified since $g \in \ell^{\infty}(X)$ satisfies the eikonal inequality and the Hardy inequality $ h\ge w $ on $ C_{c}(X) $ readily extends to $ D(H) $ (and thus we have $u \in \ell^2(X, wm)$ in the first place).  Moreover, the Caccioppoli type inequality with the assumptions of the second bullet point in Proposition~\ref{l:Cacciopolli} can be applied in the present situation. We finally point out that the integrability $u \psi \in \ell^2(X,wm)$ justifies the rearrangements of the series. This finishes the proof. 
\end{proof}

\begin{rem}
	The theorem above can be understood as an analogue to Robinson's Rellich inequality for local Dirichlet forms \cite[Theorem~1.1]{Ro18} on graphs. 		In his setting the author choses $ g=w $ and assumes an  eikonal inequality in his Condition II. Moreover, Robinson's Condition III corresponds to the  approximability assumed in~(a) of the above theorem. Furthermore, \cite{Ro18} only assumes the function $ g $ to be locally bounded rather than bounded. A corresponding assumption in the non-local setting is that $ g(x)/g(y)\leq C $ for some $ C\ge0 $ and all $ x\sim y $. By replacing a possibly unbounded $ g $ with $ g\wedge N $ for some $N \in \N$ and taking the limit $ N\to\infty $ at the end, the above proof goes through with this assumption as well.  However, for our purposes we pursue a slightly 	different line of argumentation which is carried out below.
\end{rem} 

The theorem above gives rise to our main results. There are two settings in which our results apply. The first one has an assumption on the combinatorial structure of the graph which is local finiteness. In the second setting we relax this condition to a more technical one which includes local finiteness. However, we have to assume additionally that the Hardy weight $ w $ is strictly positive.

\begin{thm}[Agmon estimate -- the locally finite case]\label{t:main} Suppose the graph is locally finite. Let   $ K\subseteq X $ be a finite set, $ w_{N}\ge 0  $ be such that  $h\ge w_{N}$ on $C_{c}(X\setminus K)$, $ N\in \N $. Suppose there is a constant $ 0<\gamma<1 $ {and a monotone increasing $(g_N)$ in $ C(X) $ converging to some $ g\ge 0 $}
	such that for all large $ N $ 
	the eikonal inequality
	\begin{align*}
		\left|\nabla g_N^{1/2}\right|^{2}\leq \gamma g_{N} w_{N} 
	\end{align*}
	is satisfied.
	Then, for every
	\begin{itemize}
		\item[(a)] approximable (positive sub-)solution   $u \in \mathcal{F}$
		\item[(b)] or solution  $u\in   D(H)$ whenever $ q\in V $
	\end{itemize}
	to the equality { $ \mathcal{H}u= f $  for some $ f\in C_{c}(X) $, one has
		\begin{align*}
			u\in \ell^{2}(X,gwm),
		\end{align*}
		where $ w=\liminf_{N\to\infty}w_{N} $. Whenever, $ w_{N}\ge w $ the statement holds also for  $f \in \ell^2(X,gw^{-1}m)$} whose support is included in the support of $ w $.	
\end{thm}

\begin{proof}  By enlarging $ K $ we can assume that  $ f\in C_{c}(X) $ is supported in $K  $. Furthermore, we argue that we can replace $ w_{N} $ by some $ w_{N}'\ge0 $ that is uniformly bounded  on $ K$. To this end, set $ w_{N}' =w_{N}$ on $X\setminus K $ and we replace $ w_{N} $ on $ K $ by 
	$$ w_N'(x)=
	\begin{cases}
		\frac{|\nabla g_{N}^{1/2}|^{2}(x)}{\gamma g_{N}(x)} &: g(x) \neq 0,\\
		0&:\mbox{else}.
	\end{cases}  $$
	Since we assumed that $ w_{N} $ satisfies  the eikonal inequality, we have $ w_{N}'\leq w_{N} $ and $ w_{N}' $ satisfies the eikonal inequality as well.
	Moreover,  since  $ w_{N}'\leq w_{N} $ we have $ h-w_{N}'\ge $  on $ X\setminus K $.  Now, for $ x\in K $, the sequence $ w_{N}'(x) $ converges to $ |\nabla g^{1/2}|^{2}(x)/(\gamma g(x))  $ whenever $ g(x) \neq 0$ and to $ 0 $ otherwise as $ N\to\infty $. Thus, the function  $ w_{N}' $ admits a uniform upper bound on the finite set $ K$.

	So, 	we have $h- w_{N}' \geq 0$ on $C_c(X \setminus K)$ and by Lemma~\ref{l:positive} there is $\chi_{N}\in C_{c}(N(K))$ such that  $h+{\chi}_{N}\ge w^{\prime}_{N}$ on $C_{c}(X)$.	
	Let $$ \tilde  \chi_{N} =\chi_{N}+1_{N(K)}\qquad\mbox{and}\qquad \tilde w_{N}=w_{N}'+ 1_{N(K)}$$ 
	for $ N\ge1 $.
	Observe that as $ h+\chi_{N}\ge w_{N}' $ on $ C_{c}(X) $ we have
	\begin{align*}
		h+\tilde \chi_{N}\ge \tilde w_{N}\qquad\mbox{on }C_{c}(X)
	\end{align*}
	and, moreover, $ \tilde w_{N}\ge 1$ on $ N(K) $.
	We define $\tilde f_{N}:X\to\R$ as 
	\[
	\tilde f_{N} := f + \tilde \chi_{N} u.
	\] 
	We  assumed that $ f $ is supported in $ K$ and since $\tilde  w_{N} $ is larger than $ 1 $ on $ N(K) $, the support of $ \tilde f_{N} $ is included in the support of $\tilde w_{N} $.
	Hence, we can apply  Theorem~\ref{p:main} to obtain
	\begin{align*}
		(1-\gamma)^{2}	\sum_{X}|u|^{2}g_{N} \tilde w_{N} m\leq \sum_{X}|\tilde{f}_{N}|^{2}g_N\tilde w^{-1}_{N}m.
	\end{align*}
	To argue that the right hand side has a uniform upper bound, we first observe that $ \tilde f_{N} $ is supported on the finite  set $ N(K) $ on which  $ g=\lim_{N\to\infty}g_{N} $ has a uniform  upper bound 
	and  $  \tilde w_{N} \geq 1 $. 
	Thus, it remains to argue that $ \tilde f_{N} $ stays bounded on $ N(K) $. To this end, we take the precise form of $ \chi_{N}=2b_{K}-\lambda_{K} $ given by Lemma~\ref{l:positive} into consideration and observe that the infimum   $ \lambda_{K} $ of $ (h-w_{N}')(\ph) $ for $\ph\in C_{c}(K)  $ is bounded from below by  the negative uniform upper bound of $ w_{N}' $ on $K$ since $ h\ge0 $. We obtain 
	\begin{align*}
		\sup_{N(K)} |\tilde \chi_{N}|=1+\sup_{N(K)} |\chi_{N}|\leq 1 +2\sup_{N(K)}b_{K}+\sup_{K}w_{N}'<\infty
	\end{align*}
	since $ N(K) $ is finite. Thus,  $ \tilde f_{N} =f+\tilde \chi_{N} u $ is uniformly bounded on $ N(K) $.
	Hence, we conclude the statement for $ f\in C_{c}(X) $ by Fatou's lemma and  $ \liminf_{N\to\infty}\tilde w_{N}=w+1_{N(K)} $.
	
	If $ w_{N}\ge w $ and $ f\in \ell^{2}(X,gw^{-1}m) $, we argue that the right hand side in the crucial inequality from Theorem~\ref{p:main} above   stays finite via Fatou on the left hand side and  monotone convergence outside of the finite set $ N(K) $ on the right hand side.
\end{proof}

We now come to the second Agmon estimate where we relax the assumption on the combinatorial structure and generalize it to a functional analytic assumption. However, here we have to assume that the Hardy weights $ w_{N} $ are strictly positive.

\begin{thm}[Agmon estimate -- strictly positive Hardy weight]\label{t:main1} Let    $ K\subseteq X $ be a finite set and $ w_{N}\ge w>0  $ be such that  $h\ge w_{N}$ on $C_{c}(X\setminus K)$, $ N\in \N $, and $ w=\liminf_{N\to\infty}w_{N}$.
Assume $ \mathcal{H}C_{c}(K)\subseteq \ell^{2}(X,gw^{-1}m) $ and that there is a constant $ 0<\gamma<1 $
and a monotone increasing  $(g_N)$ in $ C(X) $ converging to $ g\ge 0 $
	such that for all large $ N $ the eikonal inequality
	\begin{align*}
		\left|\nabla g_N^{1/2}\right|^{2}\leq \gamma g_{N} w_{N} 
	\end{align*}
is satisfied.
	Then, for every
	\begin{itemize}
		\item[(a)] approximable (positive sub-)solution   $u \in \mathcal{F}$
		\item[(b)] or solution  $u\in   D(H)$ whenever $ q\in V $
	\end{itemize}
	to the equality { $ \mathcal{H}u= f $  for some $f \in \ell^2(X,gw^{-1}m)$} one has
	\begin{align*}
		u\in \ell^{2}(X,gwm).
	\end{align*}
\end{thm}
\begin{rem} If we drop the assumption $ w>0 $, then $ \mathcal{H}C_{c}(K)\subseteq \ell^{2}(X,gw^{-1}m) $ still implies that $ w $ does not vanish on the combinatorial neighborhood of $ K $.
On the other hand given that $ w $ does not vanish on the combinatorial neighborhood of $ K $, then  $ \mathcal{H}C_{c}(K)\subseteq \ell^{2}(X,gw^{-1}m) $ is in particular satisfied if the graph is locally finite. (Local finiteness implies that $ \mathcal{H} $ maps compactly supported functions to compactly supported functions.) 
\end{rem}

\begin{proof} We consider the set $ \tilde X=X\setminus K $  and consider $ C(\tilde X) $ be a subspace of $ C(X) $ via extension by $ 0 $ and we denote $ \tilde v=v 1_{\tilde X} $ for $v\in C(X)$. The graph $ \tilde b=b\vert_{\tilde X\times \tilde X} $ gives rise to a Schrödinger operator $ \tilde{ \mathcal{H}}  $ on $ \tilde X $ defined as
	\begin{align*}
		\tilde{ \mathcal{H}} \tilde v(x)=\frac{1}{\tilde m(x)}\sum_{y\in \tilde X}\tilde b(x,y)(\tilde v(x)-\tilde v(y))+\tilde q(x)\tilde v(x),
	\end{align*}
for $ x\in \tilde X $ and  $ v\in \mathcal{F}(X) $,  where $ \tilde m=m\vert_{\tilde X} $ and $$  \tilde q(x) =q(x)+\sum_{y\in K}b(x,y). $$
Then, $ \tilde {\mathcal{H} }\tilde v(x)=\mathcal{H} v(x)$ for each $x \in \tilde{X}$ and for all  $v\in\mathcal{F}(X)  $ that vanishes on $K$. Moreover, we denote the restriction of $ h $ to $ C_{c}(\tilde X) $ by $ \tilde h $ and observe that by Green's formula, Lemma~\ref{l:Green}, we have
\begin{align*}
	\tilde h(\ph)=\sum_{\tilde X}\ph \tilde{ \mathcal{H}}\ph m
\end{align*}
for $ \ph\in C_{c}(\tilde X) $. Moreover,
we clearly have
\begin{align*}
	\tilde h\ge \tilde w_{N}
\end{align*}
on $ C_{c}(\tilde X) $.  For the  squared gradient $|\tilde{\nabla} \tilde{v}|^2 $ corresponding to $ \tilde b $ and $ \tilde m $, we have
$ 	|\tilde \nabla \tilde v|^{2}\le  | \nabla v|^{2} $
for all $ v\in C(X) $ since $ \tilde b\leq b $. Thus, as $ \tilde g_{N}\tilde w_{N}=g_{N}w_{N} $ on $ \tilde X $,
\begin{align*}
	|\tilde \nabla \tilde g_{N}^{1/2}|^{2}\leq \gamma \tilde g_{N} \tilde w_{N} \quad  \mbox{ on } \quad \tilde{X}.
\end{align*}
 Finally, the equality $ \mathcal{H}u=f $ translates to the following  equality on $ \tilde X $
\begin{align*}
	\tilde{\mathcal{H}}\tilde u= \tilde f\qquad \mbox{with }\qquad \tilde f= 1_{\tilde X} f - 1_{\tilde X} \mathcal{H} (1_{X\setminus \tilde X}u).
\end{align*}
Thus, by the Rellich inequality, Theorem~\ref{p:main}, above we conclude
\begin{align*}
	(1-\gamma)^{2}	\sum_{\tilde X}|\tilde u|^{2}\tilde g_{N} \tilde w_{N}\tilde m\leq \sum_{\tilde X}|\tilde{f}|^{2}\tilde g_N\tilde w^{-1}_{N}\tilde m
\end{align*}
for all $N \in \N$. 
By Fatou's lemma applied to the left hand side and $ w_{N}\ge w>0 $ as well as monotone convergence with respect to $ g_N \to g$ applied to the right side, we obtain	
\begin{align*}
	(1-\gamma)^{2}	\sum_{\tilde{X}}| u|^{2}\tilde g \tilde w m\leq \sum_{\tilde{X}}|\tilde{f}|^{2}\tilde g\tilde w^{-1}m.
\end{align*}	
Observe that finiteness of the left hand side is equivalent to $ u\in\ell^{2}(X,gwm) $ since $ X $ and $ \tilde X $ differ only by the finite set $ K $. Moreover, the right hand side stays finite, i.e., $\tilde f\in \ell^{2}(X,gw^{-1}m)  $ since we assumed that  $ f\in \ell^{2}(X,gw^{-1}m)  $, as well as $ \mathcal{H} (u1_{X\setminus\tilde X})\in  \ell^{2}(X,gw^{-1}m) $ by assumption since $ u1_{X\setminus\tilde X}=u1_{K} $ is finitely supported.
\end{proof}

\section{Applications}\label{s:appl}
In this  section we apply the Agmon estimates above. We  formulate the results in the locally finite case presented in Theorem~\ref{t:main}. However,  the results hold also in the setting of Theorem~\ref{t:main1} and can be formulated in this setting also quite easily which we omit for the sake of brevity.

\subsection{Agmon estimates}

First we show an Agmon type estimate using the Agmon metric. In the case of $ \R^d $ this was introduced and studied by Agmon in \cite{Ag82}.

Let a symmetric function $\si:X\times X\to [0,\infty)$ and  $w:X\to[0,\infty)$ be given. Then,
a path distance $\rho_{\si,w}:X\times X\to[0,\infty)$ is given by
\begin{align*}
\rho_{\si,w}(x,y)=\inf_{x=x_{0}\sim\ldots\sim x_{n}=y}\sum_{i=0}^{n-1}1\,\wedge(w(x_{i})\wedge w(x_{i+1}))^{\frac{1}{2}}\si(x_{i},{x_{i+1}}).
\end{align*}
This distance gives rise to an pseudo-metric on $X$ and the cut off with $ 1 $ ensures that its jumps size is bounded by $ 1 $.

The special case we are interested in is where $\si$ is given by an intrinsic metric $d$ and $w$ is a Hardy weight, i.e., $h\ge w$ in the sense of quadratic forms on $C_{c}(X)$.

\begin{thm}\label{t:main2} Suppose that the graph is locally finite {and that $d$ is an intrinsic metric}. 
	Let  $\lm\in\R$ and $w\ge 0$ be  such that $(h-\lm)\ge w$ on $C_{c}(X\setminus K)$ for some finite set $K\subseteq X$. Then, for every
\begin{itemize}
	\item[(a)] approximable (positive sub-)solution $ u\in \mathcal{F}(X) $
	\item[(b)] solution in  $ D(H)$ whenever $ q\in V $
\end{itemize}
of { $ (\mathcal{H} - \lambda)u=0 $} on $ X\setminus K $ and 
 $\rho=\rho_{d,w}(o,\cdot)$ for some fixed $o\in X$ there is $r>0$ such that 
	\begin{align*}
u\in\ell^{2}(X,e^{{r}\rho}	wm).
	\end{align*}
Specifically, one can choose $ r>0 $ such that $ r^{2}\left(\frac{1+e^{r}}{16}\right)<1 $.	
\end{thm}
We first prove a refinement of the mean value theorem for the exponential function which will serve subsequently as a substitute for the chain rule.

\begin{lemma}\label{l:estimate} For all  $ \theta:X\longrightarrow\R $,  we have with $ r=\sup_{x\sim y}|\theta(x)-\theta(y)| $
	\begin{align*}
	|\nabla e^{\theta /2}|^{2}&\leq {e^{\theta}}\left(\frac{1+e^{r}}{8}\right)|\nabla  \theta|^{2}.
	\end{align*}
\end{lemma}
\begin{proof}
  We use the elementary inequality
	$$  	
	|e^{a}-e^{b}|^{2}\leq \frac{e^{2a}+e^{2b}}{2}|a-b|^{2} $$
	that holds for all $a,b \in \R$. 
	For a proof, see e.g.\@ \cite[Lemma~2.4]{HKW} (with the choices $R=1$ and $\alpha = b-a > 0$ if $b > a$; if $b=a$ the above inequality holds  trivially). 
%
%
Then,
\begin{align*}
|\nabla e^{\theta/2}|^{2}(x)&\leq \frac{e^{\theta(x)}}{2m(x)}\sum_{y\in X}b(x,y)\left(\frac{1+e^{\theta(y)-\theta(x)}}{2}\right)\left|\frac{\theta(x)}{2}-\frac{\theta(y)}{2}\right|^{2}\\
&\leq e^{\theta(x)}\left(\frac{1+e^{r}}{8}\right)\frac{1}{2m(x)} \sum_{y\in X} b(x,y)\left|{\theta(x)}-{\theta(y)}\right|^{2}.
\end{align*}	
This finishes the proof.
\end{proof}

\begin{proof}[Proof of Theorem~\ref{t:main2}]
	For $r >0$, let $g=e^{r\rho(o,\cdot)}$ and $ g_{N}=g\wedge e^N $ for $ N\in \N $.
	Setting $  \theta ={r}\rho(o,\cdot)\wedge N$, we have $ g_{N}=e^{\theta} $. Since $\rho(x,y) \leq 1$ for $x \sim y$, we observe that $\sup_{x \sim y} |\theta(x) - \theta(y)| \leq r$. 
	By the lemma above, Lemma~\ref{l:estimate}, we estimate
	\begin{align*}
		|\nabla g_N^{1/2}|^{2}&=|\nabla e^{\theta/2}|^{2}
		\leq e^{\theta}\left(\frac{1+e^{r}}{8}\right)r^{2}|\nabla \rho(o,\cdot)|^{2} 
		\\&\leq e^{\theta}\left(\frac{1+e^{r}}{8}\right)r^{2}\frac{w}{2m}\sum_{y\in X}b(\cdot,y)d^{2}(\cdot,y)\\
		&\leq\left(\frac{1+e^{r}}{16}\right)r^{2}e^{\theta}{w}=\left(\frac{1+e^{r}}{16}\right)r^{2}g_Nw,
	\end{align*}
	where we used the intrinsic metric property of $ d $ in the last estimate.
	Thus, the statement follows by Theorem~\ref{t:main}, applied with $w = w_N$ for all~$N~\in~\N$.
\end{proof}
\begin{rem}By a numerical calculation  one sees that one can choose $ 0<r <1.6235 $, i.e., $ r=\pi/2 $ in the estimate above.	
\end{rem}

Furthermore, we prove a result which is indeed a proper analogue of Agmon's theorem on the decay of generalized eigenfunctions, cf.\@ \cite[Theorem~1.5]{Ag82} and its interpretation. It phrases growth in terms of existence of an integral and says that  solutions ``which do not grow too fast, in fact decay rapidly''.

\begin{thm}
\label{c:main2} Suppose that the graph is locally finite {and that $d$ is an intrinsic metric}. 
Let   $w\ge 0$ be  such that $h\ge w$ on $C_{c}(X)$. Assume that the space with respect to the Agmon metric $ \rho = \rho_{d,w}(o, \cdot) $ is a complete metric space, where $o \in X$ is a fixed vertex.  Then, for every (positive sub-)solution $ u\in \mathcal{F}(X) $
of { $\mathcal{H} u=f $} for some $ f\in C(X) $ for which there is $ \al>0 $
such that $ \al^{2}e^{\al}/8<1 $ and 
\begin{align*}
	u\in\ell^{2}(X,e^{-{\al}\rho}	wm)
\end{align*}
one has for $ C=(1- \al^{2}e^{\al} / 8)^{-2} $
\begin{align*}
	\sum_{X}u^{2}e^{2\al\rho}wm\leq C\sum_{X}f^{2}e^{2\al\rho }w^{-1}m.
\end{align*}
In particular, whenever $ f $ is finitely supported within the support of $ w $, then
	\begin{align*}
		u\in\ell^{2}(X,e^{{\al}\rho}	wm).
	\end{align*}
\end{thm}

\begin{proof}
The proof follows along the same idea as the proof of Theorem~\ref{p:main} but here we take explicit advantage of  having a complete Agmon metric to construct the cut-off functions.
Define   \begin{align*}
	g=g_{n}=e^{ \al(\rho\wedge(2n-\rho))/2} \quad\mbox{and}\quad \chi=\chi_{N}=\left(1- \frac{\rho(B_{N},\cdot)}{N}\right)_{+}
	\end{align*}	
\begin{align*}
			\psi=\psi_{n,N}=g_{n}\chi_{N}.
\end{align*}  
Local finiteness of the graph together with the fact that $ \rho $ is a path metric yields that the balls $ B_{N} $ are finite by a Hopf-Rinow type theorem, \cite{KMue}. 
Since the support of $ \psi_{n,N} $ is equal to the support of $ \chi_{N} $ which is $ B_{2N} $ we infer  $ \ph_{N}\in C_{c}(X) $. 
Moreover, $ g_{n}=e^{\al\rho}/2 $ on $ B_{n} $ and $ g_{n}=e^{\al(2n-\rho)/2} $ on $ X\setminus B_{n} $. Observe that for $ x\sim y $ such that $ x \in B_{n}$ and $ y\in X\setminus B_{n} $
\begin{align*}
	|\rho(x)-(2n-\rho(y))| &\leq |n - \rho(x)| + |\rho(y) -n | = |n - \rho(x) + \rho(y) -n| \\
	&= |\rho(x)-\rho(y)|\leq \rho(x,y).
\end{align*}
Thus, by the mean value theorem, the triangle inequality and the bound on the  jump size of $ \rho =\rho_{d,w} $ and the intrinsic metric property of $ d $
\begin{align*}
	|\nabla g_{n}|^{2} &\leq 
		|\nabla e^{ \al(\rho\wedge(2n-\rho))/2} |^{2}\leq\frac{\al^{2}}{8m}e^{ \al(\rho\wedge(2n-\rho)+1)} \sum_{y\in X}b(x,y) |\rho(o,x)-\rho(o,y)|^{2}\\
	&\leq \frac{\al^{2}}{8}e^{\al(\rho\wedge(2n-\rho)+1)}  w \frac{1}{m}\sum_{y\in X}b(\cdot,y)d(\cdot,y)^{2}\leq \frac{\al^{2}}{8}e^{\al} g_{n}^2 w.
	\end{align*}
	as well as
	\begin{align*}
		|\nabla \chi_{N}|^{2} \leq \frac{1}{N^{2}}\frac{1}{2m} \sum_{y\in X}b(\cdot,y)|\rho(B_{N},\cdot)-\rho(B_{N},y)|^{2}
		&\leq \frac{1}{2N^{2}}w.
	\end{align*}
	So, by the discrete Leibniz rule and estimating $ 0\leq \chi \leq 1 $ as well as Cauchy-Schwarz we obtain 
\begin{align*} 
	|\nabla \psi|^{2}=|\nabla g\chi|^{2}&\leq g^{2}|\nabla \chi|^{2}+|\nabla g|^{2}+2g|\nabla g||\nabla \chi |\\
	&\leq \left(\frac{1}{2N^{2}} +
	\frac{\al^{2}}{8}e^{\al}+\frac{\al}{4N}e^{\al/2}
	\right)g^{2}w.
\end{align*}

Retracing the estimates in the  proof of Theorem~\ref{p:main}, using the  Hardy inequality, the Caccioppoli  inequality, Lemma~\ref{l:Cacciopolli}, the Cauchy-Schwarz inequality and the estimates above, we obtain
\begin{align*}
	\sum_{X}\psi^{2}_{n,N}u^{2}wm
	\leq& \sum_{X}f\psi^{2}_{n,N}um+\sum_{X}u^2|\nabla\psi_{n,N}|^{2}m\\
	\leq&\Big(\sum_{X}u^{2}\psi^{2}_{n,N} w m\Big)^{\frac{1}{2}} \Big(\sum_{X}f^{2}\psi^{2}_{n,N}w^{-1}m\Big)^{\frac{1}{2}} \\&+\sum_{X}u^2\left(\frac{1}{2N^{2}} +
	\frac{\al^{2}}{8}e^{\al}+\frac{\al}{N}e^{\al}
	\right)g_{n}^{2}wm
\end{align*}
Now we invoke the assumption $ 	u\in\ell^{2}(X,e^{-{\al}\rho}	wm)$ to take the limit $ N\to\infty $. Indeed, $ g_{n}^2=e^{2\al n}e^{-\al\rho} $ outside of the finite set $ B_{n} $. So, we obtain by the virtue of Fatou's lemma on the left hand side and monotone convergence on the right hand sid
\begin{align*}
	\sum_{X}g^{2}_{n}u^{2}wm\leq\Big(\sum_{X}u^{2}g^{2}_{n} w m\Big)^{\frac{1}{2}} \Big(\sum_{X}f^{2}g^{2}_{n}w^{-1}m\Big)^{\frac{1}{2}} +\frac{\al^{2}}{8}e^{\al}\sum_{X}u^2
	g_{n}^{2}wm.
\end{align*}
Reordering the terms and taking the limit $ n\to\infty $ yields the result.
\end{proof}

\begin{proof}[Proof of Theorem~\ref{t:intro2}] Consider the operator $H= \Delta+q $ on a combinatorial graph with degree bounded by $ D $. The combinatorial graph distance $ d $ divided by $ \sqrt{D} $ is an intrinsic metric with finite jump size $ s=1/\sqrt{D} $.  Recalling the notion of the Agmon metric in the combinatorial situation we see that $ \varrho/\sqrt{D} =\rho_{d/\sqrt{D},w}$, where $ \varrho $ was defined in the introduction and $ w $ is a Hardy weight for $ H-\lambda $ outside of a compact set. Hence, Theorem~\ref{c:main2} above yields the statement.
\end{proof}
\subsection{Below the essential spectrum}
In this subsection we consider estimates for solutions below the essential spectrum. This in particular includes eigenfunctions of discrete eigenvalues below the essential spectrum.

\begin{thm}\label{t:main5} 
	Suppose the graph is locally finite and $d$ is an intrinsic metric with  jump size $1$. Let $\lm < \lm_{0}^{\mathrm{ess}}(H)$.
	Then, for every
	\begin{itemize}
		\item[(a)] approximable (positive sub-)solution $ u\in \mathcal{F} $
		\item[(b)] solution $ u\in D(H) $ whenever $ q\in V $
	\end{itemize}
	of $(\mathcal{H}-\lambda)u=0 $,  there is $ r>0 $  such that 
	\begin{align*}
	u\in\ell^{2}(X,e^{rd(o, \cdot )}m),
	\end{align*}
	{for some fixed $o \in X$}. 
	Specifically, we can choose $ r>0 $ such that with $	r^{2}\left(\frac{1+e^{r}}{16}\right)< \lambda_{0}^{\mathrm{ess}}(H) - \lambda =:a  $ which {in the case of $a \leq 1$} is satisfied by $ r=2ae^{-a}  $.
\end{thm}

\begin{proof}
	Let $ a=\lm_{0}^{\mathrm{ess}}(H)-\lambda $. Let furthermore, $ r>0 $ and $0< \gamma<1 $ be such that $ r^{2}\left(\frac{1+e^{r}}{16}\right)=\gamma a <a$.
	Then,   $ \lambda + \gamma a<\lm_{0}^{\mathrm{ess}}(H)$ and  by Proposition~\ref{p:AP}~(b) there is a compact set $K$ such that 
	\begin{align*}
		(h-\lm)\ge  \gamma w \quad\mbox{ on }\, C_{c}(X\setminus K),
	\end{align*}
	when we denote the constant function with value $a $ by $ w $.
	We set $g=  e^{r d(o, \cdot)}$.
	With  $  \theta =({r}d(o,\cdot))\wedge N$ we have 	$ g_N =g\wedge e^{N}= e^{\theta} $ and using that the intrinsic metric $d$ has jump size $1$ we get
	by the lemma above, Lemma~\ref{l:estimate},
	\begin{align*}
		|\nabla g_N^{1/2}|^{2} = 	|\nabla e^{\theta/2}|^{2} \leq r^{2}\left(\frac{1+e^{r}}{8}\right)e^{\theta}|\nabla d(o,\cdot)|^{2}\leq \gamma a e^{\theta} = \gamma g_N  w.
	\end{align*}
	Thus, we an apply  Theorem~\ref{t:main} to conclude the statement while noting that $w$ is a constant function. Moreover, for the particular choice of the constant $ r =2ae^{-a}$  {when $a \leq 1$} we observe that
	\begin{align*}
		r^{2}\left(\frac{1+e^{r}}{16}\right)=4a^{2} e^{-2a}\left(\frac{1+e^{2{a} e^{-a}}}{16}\right)\leq 4 a^{2} e^{-2a}\left(\frac{2e^{2{a} }}{16}\right)=\frac{a^{2}}{2}< a.
	\end{align*}
	This finishes the proof.
\end{proof}
\begin{proof}[Proof of Theorem~\ref{t:intro1}] Let the operator $H= \Delta+q $ on a combinatorial graph with degree bounded by $ D $ be given.  Furthermore, the combinatorial graph distance $ d $ divided by $ \sqrt{D} $ is an intrinsic metric with finite jump size. A generalized eigenfunction $ u $ satisfying the growth condition $ \|u\vert_{B_{n+1}\setminus B_{n-1}}\|\to 0 $, $ n\to\infty $, is approximable  due to Proposition~\ref{p:growth}.  If $ \lambda<\lambda_{0}^{\mathrm{ess}}(H) $, then $ u\in \ell^{2}(X,e^{rd(0,\cdot)/\sqrt{D}}) $ for some $ r>0 $ by the theorem above. This means that $ u $ decreases exponentially in the $ \ell^{2} $ sense.
\end{proof}

\subsection{Sparse graphs, Cheeger inequalities and discrete spectrum}

Next, we come to an application which includes sparse graphs as introduced in \cite{BGK15}, see also \cite[Section~10]{KLW} for the extension to the weighted case.  In these situations purely discrete spectrum occurs whenever the vertex degree tends to infinity uniformly. To be more precise we need some more notation.  
Let $ \deg,\deg_{m}:X\to[0,\infty) $ be given by
\begin{align*}
\deg(x)=\sum_{y\in X}b(x,y) +q_{+}(x)m(x),
\end{align*}
where $ q_{+} $ is again the positive part of $ q $,
and
\begin{align*}
\deg_{m}(x)=\frac{1}{m(x)} \deg(x)=\frac{1}{m(x)}\sum_{y\in X}b(x,y) +q_{+}(x).
\end{align*}
Furthermore, for a finite set $ W\subseteq X $, let
\begin{align*}
|\partial W|&=\sum_{(x,y)\in W\times X\setminus W}b(x,y)+\sum_{x\in W}q_{+}(x)m(x)\\|E(W)|&=\sum_{x,y\in W}b(x,y).
\end{align*}

A graph is called \emph{weakly sparse} if there are constants $ a,k \ge 0$ such that
\begin{align*}
|E(W)|\leq  a|\partial W|+km(W)
\end{align*}
for all finite sets $W \subseteq X$.
This is equivalent to the existence of constants $ \tilde a \in (0,1)$ and $ \tilde k $ such that
\begin{align}\label{e:F}\tag{F}
 (1-\tilde a) \deg_{m}-\tilde k\leq h\leq  (1+\tilde a) \deg_{m}+ \tilde k.
\end{align}
This is tightly related to so called Cheeger inequalities, i.e. if 
\begin{align*}
\alpha=\inf_{W\subseteq X \,\mbox{\scriptsize finite}}\frac{|\partial W|}{\mathrm{vol}(W)}>0, 
\end{align*}
where $ \mathrm{vol}(W)=\sum_{x\in W}\deg(x) $, then the form inequality above holds with $ \tilde a=\sqrt{1-\al^{2}} $ and $ \tilde k=0 $, \cite[proof of Proposition~15]{KL2}.
Consequently, the form inequality \eqref{e:F} implies that the operator $ H $ has purely discrete spectrum if $ \deg_{m} $ tends to infinity, i.e., 
\begin{align*}
D_{\infty}:=\liminf_{x\to\infty}\deg_{m}(x)
\end{align*}
is infinite,  where $ \infty $ denotes the additional point in   the one point compactification  $ \hat X=X\cup\{\infty\} $ of $ X $. Moreover, one also sees that the form inequality \eqref{e:F} above yields directly that for the closure of the form $h  $ one has
\begin{align}\tag{D}\label{D}
D(h)=\ell^{2}(X,m)\cap\ell^{2}(X,\deg).
\end{align}
For details we refer to \cite{BGK15,BGKLM20}. The next theorem applies to weakly sparse graphs and shows statements about the decay of eigenfunctions in this situation.


\begin{thm}\label{t:main4}
	 Suppose that the graph is locally finite.
	Let  $\lm\in\R$  and  {$a \in(0,1]$} be  such that $(h-\lm)\ge {{a}}\deg_{m}$ on $C_{c}(X\setminus K)$ for a finite set $K\subseteq X$. Then, for every
	\begin{itemize}
		\item[(a)] approximable (positive sub-)solution $ u\in \mathcal{F} $
		\item[(b)] solution  $u\in  D(H) $ whenever $ q\in V $
	\end{itemize}
	of $ (\mathcal{H}-\lambda)u=0 $ on $ X\setminus K $,
	there is $ r=r(a)>0 $ such that
	\begin{align*}
	 u\in\ell^{2}(X,e^{r|\cdot|}\deg),
	\end{align*}
	where $|x|$ is the natural graph distance of $x\in X$ to a fixed vertex $o\in X$.  In particular, one can choose $ r>0 $ such that {$	r^{2}\left(\frac{1+e^{r}}{16}\right) <a  $},  for example $ r=2a e^{-a} $.
\end{thm}
\begin{proof}
	For $r > 0$ and $N \in \N$, define 	 $$   \theta =N\wedge{r}|\cdot| \quad\mbox{and}\quad g_N = e^{N}\wedge g=e^{\theta}.  $$ 
	{Now, choose $r > 0$ and set $\gamma= 	a^{-1}r^{2}\left(\frac{1+e^{r}}{16}\right) < 1$ so that $\gamma a < a$}. Then, denoting the combinatorial graph distance by $ d_{n} $ we obtain  $ |\theta(x)-\theta(y)| \leq r d_{n}(x,y)\leq r$ for $x \sim y$, and by Lemma~\ref{l:estimate} we find
	\begin{align*}
		\left|\nabla g_N^{1/2}\right|^2 =	\left|\nabla e^{\theta/2}\right|^{2}
		\leq 2\gamma {a} e^{\theta} \big|\nabla |\cdot| \big|^{2} = {\gamma} {a} e^{\theta} \deg_{m} = {\gamma a g_N  \operatorname{deg}_m}.
	\end{align*}
	Hence, the result follows  by Theorem~\ref{t:main}, applied with $w = w_N$ for all $N \in \N$. 
	The statement about the possible choice of $ r $ follows analogously to the argument given in the proof of the previous statement {while noting that ${a} \leq 1$}.
\end{proof}

The  Cheeger constant at infinity be given by
\begin{align*}
\al_{\infty}=\sup_{{K\subseteq X \mbox{\scriptsize\, finite}} }\inf_{W\subseteq X\setminus K \mbox{\scriptsize\, finite } }\frac{|\partial W|}{\mathrm{vol}(W)},
\end{align*}
where $ |\partial W|$  and  $\mathrm{vol}(W) $ are defined above. Also recall 
$ D_{\infty}=\liminf\deg_{m}(x) $. Whenever $ \alpha_{\infty}>0 $ and $ D_{\infty}=\infty $, then $ H $ has purely discrete spectrum, confer  \cite[Theorem~20]{KL2}. Thus, the next theorem gives estimates on the decay of the eigenfunctions in this case. \smallskip

\begin{thm}Suppose a locally finite graph is given with $ q=0 $ which  has positive Cheeger constant at infinity $\al_{\infty}>0$ and ${D}_{\infty}=\infty$. {Let $\lambda \in \R$} and let $ d $ be an intrinsic metric with jump size  $ 1 $.
	Then, for every
\begin{itemize}
	\item[(a)] approximable (positive sub-)solution $ u\in \mathcal{F} $
	\item[(b)] solution  $u\in  D(H)$
\end{itemize}
of $( \mathcal{H}-\lambda)u=0 $ and  all $R>0$, we have
	\begin{align*}
u\in \ell^{2}(X,	e^{Rd(o,\cdot)}\deg)
	\end{align*}
	and there is $ r=r(\al_{\infty})>0 $ such that
	\begin{align*}
u\in\ell^{2}(X, e^{r|\cdot|}\deg),
	\end{align*}
	where  $|x|$ is the natural graph distance of $x\in X$ to a fixed vertex $o\in X$.  Specifically, one can choose $ r>0 $ such that $	r^{2}\left(\frac{1+e^{r}}{16}\right)<1 -  \sqrt{1- \alpha_{\infty}^2}=:a$  which is for instance satisfied by $ r=2a e^{-a} $.
\end{thm}
\begin{proof}	
	Let $\eps>0$. Choose $n=n(\eps)>0$  such that
	$\lambda  < {\varepsilon a {n}}(\log n)^2/2$,
	where $a =1 -  \sqrt{1- \alpha_{\infty}^2}$.
	Choose  $K\subseteq X$ finite such that
	$$\deg_m \geq n (\log n)^2\qquad\mbox{ on } X\setminus K$$
	and
	\begin{align*}
		h\ge (1-\eps/2) a\deg_m \qquad\mbox{ on } C_{c}(X\setminus K)
	\end{align*}
	which is possible by   \cite[Proposition~14 and proof of Proposition~15]{KL2} or \cite[Theorem~5.1 or proof of Theorem~5.3]{BGK15}.
	Then,
	\begin{align*}
		h-\lm\ge (1-\eps/2) a\deg_m -\lm\ge
		(1-\eps)a\deg_m =:w
	\end{align*}
	on $C_c(X \setminus K)$. 
	Hence, the second statement follows from Theorem~\ref{t:main4}. \smallskip
	
	For the first statement, let $R=\sqrt{(1-\eps)a}\log n$ and $ g=e^{Rd(o,\cdot)} $. 
	With  $  \theta =({R}d(o,\cdot))\wedge N$ for $N \in \N$ and, therefore, with	$g_N =  g\wedge e^{N}=e^{\theta} $ we have by Lemma~\ref{l:estimate} above {(which is applicable since $d$ has jump size $1$)}, the intrinsic metric property, $ R\leq \log n $ and $\mathrm{deg}_m \geq n(\log n)^2$ outside of $K$, 	
	\begin{align*}
		|\nabla g_N^{1/2}|^2 = |\nabla e^{\theta/2}|^{2}& \leq R^{2}\left(\frac{1+e^{R}}{8}\right)e^{\theta}|\nabla d(o,\cdot)|^{2}\\
		&\leq(1-\eps)a (\log n)^{2}\frac{1+n}{16}e^{\theta}\\
		&\leq \frac{1}{8}e^{\theta}(1-\eps)a\deg_{m}  =\frac{1}{8}e^{\theta}w = \frac{1}{8}g_N w.
	\end{align*}
	This and the fact that $n$ can be made arbirtrarily large allow us to obtain the first statement from Theorem~\ref{t:main}, applied with $w = w_N$ for all $N \in \N$. 
\end{proof}

\subsection{Hardy inequalities and the supersolution construction}

In this section we prove Agmon estimates which rely on Hardy inequalities obtained via the so called supersolution contruction. This method was introduced in \cite{DFP,DP16} for elliptic operators in the continuum and later extended to graphs in \cite{KePiPo2}. It was also used in \cite{KPP_Rellich} to prove Rellich inequalities.

We shortly recall the supersolution construction to obtain Hardy weights and, moreover, Rellich inequalities. We restrict ourselves to the case of $ q=0 $.

Let $ v$ be a strictly
positive superharmonic function. Then, for $ \al\in (0,1] $ the function
\begin{align*}
w_{\al}=\frac{\mathcal{H}(v^{\al})}{v^{\al}}
\end{align*}
is a Hardy weight. For $ \alpha=1/2 $, we call
\begin{align*}
w=w_{1/2}=\frac{\mathcal{H}v^{1/2}}{v^{1/2}}=\frac{1}{2v^{1/2}}\mathcal{H}v+\frac{|\nabla (v^{1/2} )|^{2}}{v}
\end{align*}
the Hardy weight corresponding to $ v $. Under some additional assumptions $ w $ is optimal. Optimality  means that
\begin{itemize}
	\item $\mathcal{H}-w $ is {\em critical}, i.e., for all Hardy weights $ w' $ with $ w'\ge w $ we have $ w=w' $.
	\item $\mathcal{H}-w$ is {\em null-critical with respect to} $w$, i.e.\@ the ground state $u^{1/2}$ of $ \mathcal{H}-w $ is not in $ \ell^{2}(X,wm) $.
\end{itemize}
The aforementioned additional assumptions for optimality are that $ v$ is harmonic outside of a finite set and $ v $ is  proper and satisfies an bounded oscillation  condition
\begin{align*}
\sup_{x\sim y}\frac{v(x)}{v(y)}<\infty.
\end{align*}

In the following theorem we deal with the Hardy weight corresponding to a strictly positive superharmonic function $v$ that only has to satisfy the bounded oscillation condition. 

\begin{thm}\label{t:mainHardy}	Suppose the graph is locally finite  and $ q=0 $. Let $v $ be a strictly positive superharmonic function that satisfies the bounded oscillation condition $\sup_{x \sim y} v(x) / v(y) < \infty$ and let $ w $ be the corresponding Hardy weight, i.e., $ w=\mathcal{H}(v^{1/2})/v^{1/2} $. 	Then, for every
	\begin{itemize}
		\item[(a)] approximable (positive sub-)solution $ u\in \mathcal{F} $
		\item[(b)] solution {$u \in D(H)$}
	\end{itemize}
	of {$\mathcal{H}u= f $ for some $f \in C_c(X)$} and for all $ \al\in (0,1) $
	\begin{align*}
	u\in\ell^{2}(X,v^{\alpha}wm).
	\end{align*}
\end{thm}

\begin{proof}
	Fix $0 < \alpha < 1$. 
			Observe that for every $ N \in \N $ the function $v_N := v \wedge N $ is a strictly  positive superharmonic function. We denote the Hardy weight corresponding to $ v_N$ by $ w_{N} $, i.e.\@ $w_N := \mathcal{H} v_N^{1/2}/v_N^{1/2}$. Setting 
			$g_N := (v \, \wedge \, N)^{\alpha}$ we claim that 
			\begin{align*}
				\frac{\big| \nabla g_N^{1/2} \big|^2}{g_N} \leq \gamma \, w_N.
			\end{align*} 	
	To this end, we observe first that 
	since $v$ satisfies the bounded oscillation condition, we find some $0 < \varepsilon_0 < 1$ such that $\inf_{x \sim y} v(x)/v(y) \geq \varepsilon_0^2$. It is easy to see that this implies 
	\begin{align*}
	\inf_{x \sim y} \frac{v_N(x)}{v_N(y)} \geq \varepsilon_0^2
	\end{align*}
	as well. We now set $\gamma:= \gamma(\varepsilon_0, \alpha):= \Big( \frac{1- \varepsilon_0^{\alpha}}{1-\varepsilon_0} \Big)^2 < 1$.	
	It follows from \cite[Lemma~4.2]{KPP_Rellich} that for all $a \geq \varepsilon_0$ and all $t > 0$, we have 
	\begin{align*}
	\frac{\big| t^{\alpha} - (at)^{\alpha} \big|^2}{t^{2\alpha}} \leq \gamma \, \big( 1-a \big)^2 = \gamma\, \left( \frac{t - at}{t} \right)^2.
	\end{align*}
	For $x \in X$, we let $t := v_N^{1/2}(x)$ and  observe that $a:= (v_N(y)/ v_N(x))^{{1/2}} \geq \varepsilon_0$ for all $y \in X$ with $y \sim x$, which holds due to the bounded oscillation condition derived above for $v_N$. Applying now the above inequality we find 
	\begin{align*}
	\frac{\big| \nabla g_N^{1/2}(x) \big|^2}{g_N(x)} \leq \gamma \, \frac{\big| \nabla v_N^{1/2}(x) \big|^2}{v_N(x)} \leq \gamma \, w_N(x). 
	\end{align*} 
	Since the graph is locally finite we obtain the statement from Theorem~\ref{t:main}, after observing that $\lim_{N \to \infty} g_N = v^{\alpha}$ monotonically increasingly and \\ 
	$\lim_{N \to \infty} w_N =w$.
\end{proof}
For the case of general $ q $ one needs another harmonic function $ v' $ such that $ v'/v $ is proper and of bounded oscillation. Then, one can consider  the corresponding Hardy weight $ w=\mathcal{H}(vv')^{1/2}/(vv')^{1/2} $.\medskip

\textbf{Acknowledgments.} The authors acknowledge the funding by the German Science foundation. We are endebted to Yehuda Pinchover for numerous discussions and for him generously sharing his knowledge on the subject. Moreover, we thank David Damanik and Sylvain Gol\'enia for most valuable comments.
\bibliography{literature}
\bibliographystyle{alpha}

\end{document}